\documentclass[a4paper,10pt,reqno]{amsart}

\usepackage{amssymb}
\usepackage{amsmath}
\usepackage{amsthm}
\usepackage{amsfonts}
\usepackage[english]{babel}
\usepackage[utf8]{inputenc}

\usepackage[margin=3cm]{geometry}

\usepackage[unicode,colorlinks,plainpages=false,hyperindex=true,bookmarksnumbered=true,bookmarksopen=false,pdfpagelabels]{hyperref}
\hypersetup{urlcolor=cyan,linkcolor=blue,citecolor=red,colorlinks=true}
\usepackage{color}

\makeatletter
\renewcommand*{\p@subsection}{\S\,}
\renewcommand*{\p@subsubsection}{\S\,}
\makeatother

\theoremstyle{plain}
\newtheorem{theorem}{Theorem}[section]
\newtheorem{lemma}[theorem]{Lemma}
\newtheorem{proposition}[theorem]{Proposition}
\newtheorem{corollary}[theorem]{Corollary}

\newtheorem{thmA}{Theorem}

\theoremstyle{definition}
\newtheorem{definition}[theorem]{Definition}

\theoremstyle{remark}
\newtheorem{remark}[theorem]{Remark}

\numberwithin{equation}{section}

\newcommand{\CC}{\ensuremath{\mathbb{C}}}
\newcommand{\Q}{\ensuremath{\mathbb{Q}}}
\newcommand{\Z}{\ensuremath{\mathbb{Z}}}
\newcommand{\Xtt}{\ensuremath{\mathtt{X}}}

\newcommand{\tr}{\operatorname{tr}}
 
\newcommand{\Hom}{\operatorname{Hom}}
\newcommand{\End}{\operatorname{End}}

\newcommand{\Mat}{\operatorname{Mat}}
\newcommand{\id}{\operatorname{id}}
\newcommand{\Id}{\operatorname{Id}}
\newcommand{\Gl}{\operatorname{GL}}
\newcommand{\Rep}{\operatorname{Rep}}

\newcommand{\Sym}{\operatorname{Sym}}
\newcommand{\Ad}{\operatorname{Ad}} 

\newcommand{\gl}{\ensuremath{\mathfrak{gl}}}
\newcommand{\g}{\ensuremath{\mathfrak{g}}}
\newcommand{\h}{\ensuremath{\mathfrak{h}}}
\newcommand{\del}{\ensuremath{\partial}}
\newcommand{\nfat}{\ensuremath{\mathbf{n}}} 
\newcommand{\qfat}{\ensuremath{\mathbf{q}}}
\newcommand{\Mcal}{\ensuremath{\mathcal{M}}} 
\newcommand{\Ccal}{\ensuremath{\mathcal{C}}} 
\newcommand{\zbar}{\ensuremath{\underline{z}}}
\newcommand{\dd}{\ensuremath{\mathrm{d}}}
\newcommand{\Inf}{\ensuremath{\mathrm{Inf}}}

\newcommand{\reg}{\ensuremath{\mathrm{reg}}} 
\newcommand{\loc}{\ensuremath{\mathrm{loc}}}

\newcommand\br[1]{\{ #1 \}}
\newcommand\brVdB[1]{\{ #1 \}_{\mathtt{VdB}}}
\newcommand\dgal[1]{  \left\{\!\!\left\{#1\right\}\!\!\right\} }

\newcommand{\Addresses}{{
  \bigskip
  \footnotesize

\noindent M.~Fairon, \textsc{Université Bourgogne Europe, CNRS, IMB UMR 5584, F-21000 Dijon, France}

\noindent  \textit{E-mail address:}  \texttt{maxime.fairon@u-bourgogne.fr} 
}}

\title[Compatible Poisson Structures on MQV]{Compatible Poisson Structures \\ On Multiplicative Quiver Varieties}
\author{Maxime Fairon}

\begin{document}

 \begin{abstract}
Any multiplicative quiver variety is endowed with a Poisson structure constructed by Van den Bergh through reduction from a Hamiltonian quasi-Poisson structure. The smooth locus carries a corresponding symplectic form defined by Yamakawa through quasi-Hamiltonian reduction.
In this note, we include the Poisson structure as part of a pencil of compatible Poisson structures on the multiplicative quiver variety. 
The pencil is defined by reduction from a pencil of Hamiltonian quasi-Poisson structures which has dimension $\ell(\ell-1)/2$, where $\ell$ is the number of arrows in the underlying quiver. 
For each element of the pencil, we exhibit the corresponding compatible symplectic or quasi-Hamiltonian structure. 
We comment on analogous constructions for character varieties and quiver varieties. 
This formalism is applied to the spin Ruijsenaars-Schneider phase space in order to explain the compatibility of two Poisson structures that have recently appeared in the literature.
 \end{abstract} 

\maketitle


\setcounter{tocdepth}{1}
\tableofcontents


\section{Introduction} \label{S:Intro}

\subsection*{Geometric perspective} 

Consider a double quiver $\overline{Q}$, that is a directed graph equipped with an involution $a\mapsto a^\ast$ on its arrows that reverses their orientation (cf. \ref{ss:Not} for the notations). 
Denoting the vertex set by $I$, and fixing a dimension vector $\nfat\in \Z_{\geq 0}^I$,  we can construct the corresponding \emph{representation space} $M_{\overline{Q}}(\nfat)$. 
It is a smooth affine variety parametrized by points that associate a matrix $\Xtt_a\in \End(\oplus_{s\in I} \CC^{n_s})$ with each arrow $a\in \overline{Q}$, such that $\Xtt_a$ has only one non-trivial block of size $n_{t(a)}\times n_{h(a)}$ if $a$ goes from the vertex $t(a)$ (the tail) to the vertex $h(a)$ (the head), see \ref{ss:Not} and \ref{ss:MQV} for precise conventions. 
There is a natural action of $\Gl_{\nfat}:=\prod_{s\in I} \Gl_{n_s}(\CC)$ on $M_{\overline{Q}}(\nfat)$ by simultaneous conjugation of the matrices parametrizing a point $\Xtt:=(\Xtt_a)_{a\in \overline{Q}}$. 
 
Define the open subvariety  $M_{\overline{Q}}^\circ\subset M_{\overline{Q}}$ (we omit to write $\nfat$ hereafter) by imposing the condition $\det(\Id+\Xtt_a \Xtt_{a^\ast})\neq 0$ for each $a\in \overline{Q}$, where $\Id$ is the identity on $\oplus_{s\in I} \CC^{n_s}$. Then the action of $\Gl_{\nfat}$ restricts to $M_{\overline{Q}}^\circ$, and it makes the following morphism equivariant 
\begin{equation} \label{Eq:Intro1}
\Phi:M_{\overline{Q}}^\circ\to \Gl_{\nfat},\quad 
\Phi(\Xtt):=\prod_{a \in \overline{Q}} (\Id+\Xtt_a \Xtt_{a^\ast})^{\epsilon(a)}\,.
\end{equation}
Fix $\qfat=(q_s)\in (\CC^\times)^I$ and form the closed subvariety $\Phi^{-1}(\prod_s q_s \Id_{n_s})$, which first appeared in the study of the Deligne-Simpson problem by Crawley-Boevey and Shaw \cite{CBShaw}. 
In particular, it was shown that $\Phi^{-1}(\prod_s q_s \Id_{n_s})$ is, up to isomorphism, independent of the ordering of the factors in \eqref{Eq:Intro1}, or the orientation $\epsilon:\overline{Q}\to \{\pm1\}$ satisfying $\epsilon(a^\ast)=-\epsilon(a)$. 
The corresponding moduli space $\Mcal_{\overline{Q},\qfat,\theta}$ of stability $\theta\in \Q^I$ with respect to the $\Gl_{\nfat}$ action is a  \emph{multiplicative quiver variety},  see Definition \ref{Def:MQV}. 
As the name suggests, these are analogues of (Nakajima or additive) quiver varieties~\cite{Nak}. 
They attracted attention for studying: their properties of normality~\cite{KS}, the generators of their pure cohomology~\cite{MGN},  their quantization~\cite{GJS,Jo}, their symplectic resolutions \cite{ST}; or to provide their realization as: partial compactifications of character varieties \cite{CB13,Ya}, wild character varieties \cite{Bo15}, moduli of microlocal sheaves on a curve \cite{BK}, phase spaces of integrable systems of Ruijsenaars-Schneider type~\cite{BEF,CF,Fa} (and references therein).

Leaving aside these various directions, we are interested in the original question regarding multiplicative quiver varieties. 
Namely, the problem was to endow $\Mcal_{\overline{Q},\qfat,\theta}$ with a Poisson (or maybe symplectic) structure, such that $\Phi$ \eqref{Eq:Intro1} can be interpreted as a kind of moment map. 
This was solved by Van den Bergh using quasi-Poisson reduction \cite{VdB1}, or by Van den Bergh and independently by Yamakawa using quasi-Hamiltonian reduction \cite{VdB2,Ya}. 
There is a belief that the Poisson structure hence obtained on $\Mcal_{\overline{Q},\qfat,\theta}$ is unique\footnote{Our point of view is to consider the existence of various Poisson structures for the \emph{same} complex structure on the variety. The present paper does not aim at endowing multiplicative quiver varieties with a hyperk\"{a}hler structure, which is an open problem whose solution is only known in some cases, see \cite{Bo15}.} (up to obvious rescaling), as one would generally talk about ``the'' Poisson structure on this variety (or ``the'' symplectic form on its smooth locus). 
Hence, the present work was initiated to investigate this belief, and to build examples where uniqueness does not hold. 
To do so, we shall take a step back and construct pencils of compatible quasi-Poisson structures on $M_{\overline{Q}}^\circ$ before reduction (cf. \ref{s:PenGen} for precise definitions). There, one can rely on a widespread method for studying quiver varieties which is to consider, for each arrow $a\in \overline{Q}$, the (non-contracting) $\CC^\times$-action by symplectomorphisms    
\begin{equation} \label{Eq:Intro2} 
 \Xtt_a \mapsto \lambda \, \Xtt_a, \quad 
 \Xtt_{a^\ast} \mapsto \lambda^{-1} \Xtt_{a^\ast}, \quad 
 \Xtt_b \mapsto \Xtt_b \text{ if }b\neq a,a^\ast, \quad \lambda \in \CC^\times.
\end{equation}
Similarly, the mapping \eqref{Eq:Intro2} preserves Van den Bergh's quasi-Poisson structure on  $M_{\overline{Q}}^\circ$. 
Taking exterior products of infinitesimal actions associated with \eqref{Eq:Intro2} and adding them to Van den Bergh's quasi-Poisson bivector, it is a striking (yet simple) observation that we still have a quasi-Poisson structure with moment map \eqref{Eq:Intro1}. 
A more technical construction also allows to give the corresponding $2$-form. 
This leads to the first main result, proved in \ref{ss:PencilMQV}, for a quiver $Q$ whose double is $\overline{Q}$ (the number of arrows in $Q$ is denoted $|Q|$). 

\begin{thmA} \label{Thm:Main1}
Set $r_Q:=\frac12|Q|(|Q|-1)$. Then the $\Gl_{\nfat}$-variety $M_{\overline{Q}}^\circ$ admits a Hamiltonian quasi-Poisson pencil of order $r \leq r_Q$ centered at the quasi-Poisson bivector $P$ \eqref{Eq:qP-MQV} constructed by Van den Bergh \cite{VdB1}, with moment map  $\Phi$ \eqref{Eq:Intro1}. 
If $n_s\geq 1$ for each $s\in I$, the order is $r=r_Q$. 

Furthermore,  each element from the pencil is non-degenerate, and we can explicitly write the corresponding quasi-Hamiltonian $2$-form. 
By reduction, this descends to a Poisson pencil on the associated multiplicative quiver varieties.
\end{thmA}
Interestingly, this method can be applied in more general contexts, cf. the key Proposition \ref{Pr:Pencil-Cact}. 
We are able to state variants of Theorem \ref{Thm:Main1} in \ref{s:VarPencil} for generalized multiplicative quiver varieties, character varieties, or (additive) quiver varieties. 
We shall use Theorem \ref{Thm:Main1} in Section \ref{S:ApplRS} to study a family of multiplicative quiver varieties with numerous non-equivalent Poisson structures, and whose motivation we now review.

\subsection*{Mathematical physics perspective} 

In 1995, a spin generalization of the celebrated Ruijsenaars-Schneider (RS) system \cite{RS} was introduced by Krichever and Zabrodin \cite{KrZ}. Loosely speaking, they managed to upgrade a classical integrable system of $n$ relativistic particles by endowing each particle with $d\geq 2$ internal degrees of freedom, the so-called \emph{spin variables}, without breaking integrability of the system. In fact, their approach used an extra reduction of the model to claim integrability, which thus left open the question of describing the phase space of the system as a complex Poisson manifold of (complex) dimension $2nd$. 
In the case of a rational potential, this was quickly solved by Arutyunov and Frolov \cite{AF}. 
However, apart from a conjecture on the form of the Poisson brackets in the trigonometric case formulated in \cite{AF}, and the $n=2$ elliptic case partly studied in \cite{So}, no progress could be achieved for almost 20 years.   

Using the technology of multiplicative quiver varieties, Chalykh and the author \cite{CF} described the phase space of the spin RS system in the trigonometric case from the quiver $Q_d$ made of $1$ loop and $d$ extra arrows. 
This approach is satisfying as it gives a direct multiplicative analogue of Wilson's construction for the (non-relativistic) spin Calogero-Moser system \cite{T15,Wi}. Furthermore, it provided a proof to the conjecture of Arutyunov and Frolov \cite{AF}. 
While this may have closed the study of the trigonometric spin RS system, a preprint by Arutyunov and Olivucci \cite{AO} appeared slightly later, where they derived the phase space through Poisson-Lie reduction. 
This last work created some puzzlement because it constructs another Poisson structure which \emph{does not} satisfy the original conjecture from \cite{AF}. 
It may have been unclear how to resolve that state of affairs without a key observation about a real form of the trigonometric spin RS system recently considered in~\cite{FF23}: the real form admits a pencil of compatible Poisson brackets, which all yield the \emph{same} equations of motion. It seemed therefore natural to conjecture in \cite{FF23} that a similar situation arises in the complex setting, and that it bridges the gap between the approaches of \cite{CF} and \cite{AO}.  
The second main result of this paper meets this expectation as follows.
\begin{thmA} \label{Thm:Main2} 
Conjecture~C.2 in \cite{FF23} is true.
That is, the Poisson structures constructed by Chalykh-Fairon \cite{CF} and Arutyunov-Olivucci \cite{AO} on a local parametrization of the spin RS phase space belong to a Poisson pencil of order $\frac{(d-1)(d-2)}{2} \leq r \leq \frac{(d-1)d}{2}$, hence they are compatible.
Furthermore, any non-degenerate Poisson bracket from this Poisson pencil provides a Hamiltonian formulation for the trigonometric spin RS system.
\end{thmA}
The first part of the statement is obtained in Proposition \ref{Prop:pencilRS}, while the second part is established in Proposition \ref{Prop:HamRS}. To derive the pencil, we explicitly rely on Theorem \ref{Thm:Main1} to construct a quasi-Poisson pencil before quasi-Hamiltonian reduction on an open representation space $M_{\overline{Q}_d}^\circ$ associated with the quiver $Q_d$.  

We should warn the reader that we are working in a local coordinate system to prove Theorem \ref{Thm:Main2} because the reduction techniques are different: for multiplicative quiver varieties we use quasi-Poisson reduction, while the approach of Arutyunov-Olivucci uses Poisson-Lie reduction. In particular, the master phase space $M_{\overline{Q}_d}^\circ$ and its moment map do \emph{not} appear in \cite{AO}; only the slice $\Phi^{-1}(q\Id_n)$ with $\Phi$ given by \eqref{Eq:Momap-RS-2} can be realized in \cite{AO} using the ($2^n$ to $1$) local diffeomorphism $\Gl_n(\CC)^\ast \to \Gl_n(\CC)$, $(h_+,h_-)\mapsto h_+ h_-^{-1}$, where $\Gl_n(\CC)^\ast$ is the Poisson-Lie dual of $\Gl_n(\CC)$. This slice does not carry an induced Poisson structure, for the bracket is only closed with respect to invariant functions on the slice.   
Thus, one is forced to compare the 2 models only after reduction.

\medskip 

\noindent \textbf{Layout.} In Section~\ref{S:Rem}, we recall the basics about algebraic quasi-Poisson and quasi-Hamiltonian geometries, and their use for constructing multiplicative quiver varieties following Van den Bergh and Yamakawa \cite{VdB1,VdB2,Ya}. 
We introduce the notion of quasi-Poisson pencils in Section~\ref{S:Pencil}, and we describe their elementary properties and ways to build these. Such constructions are applied to the open representation spaces $M_{\overline{Q}}^\circ$
that lead to multiplicative quiver varieties, hence proving Theorem~\ref{Thm:Main1}. We also consider this formalism for related families of moduli spaces, such as character varieties and quiver varieties. 
We investigate a universal formulation of the pencil in Section~\ref{Sec:NC}, in the spirit of Van den Bergh's noncommutative Poisson geometry \cite{VdB1}.  
Section~\ref{S:ApplRS} contains the parts of the paper related to the spin Ruijsenaars-Schneider phase space where, in particular, we derive Theorem~\ref{Thm:Main2}. 

\medskip 

\noindent \textbf{Acknowledgments.}
It is my pleasure to thank Oleg Chalykh, Laszlo Feh\'{e}r, David Fern\'{a}ndez, Anne Moreau and Olivier Schiffmann for stimulating discussions. I am especially indebted to Laszlo Feh\'{e}r for comments on a first version of this paper, and to Anton Alekseev for encouraging me to develop Remark \ref{Rem:MixedPoi} and Section \ref{Sec:NC}.
This project received funding from the European Union's Horizon 2020 research and innovation programme under the Marie 
Sk\l{}odowska-Curie grant agreement No.~101034255, while affiliated to \emph{Laboratoire de Mathématiques d'Orsay}.
It was completed at IMB which is supported by the EIPHI Graduate School (contract ANR-17-EURE-0002).


\section{Original Poisson structure on multiplicative quiver varieties}  \label{S:Rem}

We recall the construction of multiplicative quiver varieties and their known Poisson structure.

\subsection{Notation}  \label{ss:Not}

We work over the field of complex numbers $\CC$ and set $\otimes := \otimes_\CC$. 

A quiver $(Q,I,h,t)$, or $Q$ for short, is the data of a set $Q$ of arrows, a set $I$ of vertices, and the head or tail maps $h,t:Q\to I$. Quivers are assumed to be finite, i.e. $|Q|,|I|<\infty$. We may depict an arrow $a\in Q$ as $a:t(a)\to h(a)$ or $t(a) \stackrel{a}{\longrightarrow} h(a)$. 
The path algebra $\CC Q$ of $Q$ is the $\CC$-algebra generated by $\{a\in Q\}\cup \{e_s \mid s\in I\}$, where the latter are a complete set of orthogonal idempotents (i.e. $e_r e_s=\delta_{rs} e_s$ for $r,s\in I$ and $\sum_{s\in I} e_s = 1$), subject to the mixed relations $a=e_{t(a)}ae_{h(a)}$ for each $a\in Q$. 
This means that we write path from left to right as in \cite{VdB1,VdB2}. 
We put $B:=\oplus_{s\in I} \CC e_s$ and see $\CC Q$ as a $B$-algebra.

Let $Q$ be a quiver. 
Denote by $\overline{Q}$ its double, obtained by adding an arrow $a^\ast:h(a)\to t(a)$ for each $a\in Q$. 
This endows $\overline{Q}$ with an involution $(-)^\ast:\overline{Q}\to \overline{Q}$ if we put $(a^\ast)^\ast:=a$ for each $a\in Q$.
Thus $h(a^\ast)=t(a)$ for all $a\in \overline{Q}$.
To distinguish between arrows originally in $Q$ and those in $\overline{Q}\setminus Q$, we have the orientation map $\epsilon : \overline{Q}\to \{\pm 1\}$ given by $\epsilon(a)=+1$, $\epsilon(a^\ast)=-1$ for any $a\in Q$.

\subsection{Quasi-Poisson and quasi-Hamiltonian geometry} 

The complex algebraic treatment of quasi-Hamiltonian \cite{AMM} and quasi-Poisson \cite{AKSM} geometry can be found in Boalch \cite{Bo07}, Van den Bergh \cite{VdB1,VdB2} and Yamakawa \cite{Ya}. The recent work of Huebschmann \cite{Hu} provides a more general setup. 

\subsubsection{Basics} \label{ss:Setup}
Hereafter, we assume that $G$ is a complex reductive algebraic subgroup of $\Gl_N(\CC)$ for some $N\geq 1$, and its Lie algebra  $\g$, seen as a subspace of $\gl_N(\CC)$, inherits the $\Ad$-invariant non-degenerate symmetric trace form 
\begin{equation*}
 (-,-)_{\g} : \g\times\g\to \CC\,, \quad (\xi,\xi')_{\g}:= \tr(\xi \xi')\,.
\end{equation*}
Fix dual bases $(E_a)_{a}$ and $(E^a)_a$ of $\g$ with respect to the trace form, i.e. $(E_a,E^b)_{\g}=\delta_{ab}$.
We define the $\Ad$-invariant Cartan trivector  $\phi \in\wedge^3\g$ by
\begin{equation} \label{Eq:Cartan3}
 \phi :=\frac{1}{12} \sum_{a,b,c} (E^a,[E^b,E^c])_{\g} \, E_a\wedge E_b \wedge E_c\,.
\end{equation}
Introduce for any $\xi \in \g$ the left- and right-invariant vector fields $\xi^L$ and $\xi^R$ on $G$.
They act by derivation according to  
$\xi^L(F)(g)=\left.\frac{d}{dt}\right|_{t=0} F(g \,e^{t \xi})$ and 
$\xi^R(F)(g)=\left.\frac{d}{dt}\right|_{t=0} F( e^{t \xi}\, g)$, 
for any function $F$ in an open neighborhood of $g\in G$. 
(We shall freely use analytic notation.) 
They allow to define a $\g$-valued differential operator $\mathcal{D}$ by $\mathcal{D}(F):=\frac12 \sum_a (E_a^L+E_a^R)(F)\, E^a$ for any function $F$ on $G$. 

Let $\theta^L=g^{-1}\dd g$ and $\theta^R=\dd g\, g^{-1}$ be the left- and right-invariant Maurer-Cartan elements in matrix notation. 
These are $\g$-valued $1$-forms, from which we define the closed $\Ad$-invariant $3$-form
\begin{equation} \label{Eq:eta3}
 \eta:=\frac{1}{12} (\theta^R\stackrel{\wedge}{,}[\theta^R\stackrel{\wedge}{,}\theta^R])_{\g}
 =\frac{1}{6} \tr(\dd g \, g^{-1} \wedge \dd g\, g^{-1} \wedge \dd g\, g^{-1})\,.
\end{equation}

Let $M$ be an affine complex algebraic variety with a left action of $G$ denoted by $(g,x)\mapsto g\cdot x$ for any $g\in G$, $x\in M$. 
The infinitesimal action of $\g$ on $M$ assigns to each  $\xi\in \g$ a vector field $\xi_M$ on $M$ satisfying 
\begin{equation} \label{EqinfVectM}
 \xi_M(f)(x)=\left.\frac{d}{dt}\right|_{t=0} f(\exp(-t\xi)\cdot x)\,,
\end{equation}
for any function $f$ in an open neighborhood of $x\in M$.
The map $\xi\mapsto \xi_M$ can be extended equivariantly to return a $k$-vector field $\psi_M$ for any $\psi\in\wedge^k\g$.  
We denote by $Q_x\in \bigwedge^\bullet T_xM$ the evaluation of a multivector field $Q$ on $M$ at a point $x\in M$ and use a similar notation for $k$-forms.

\subsubsection{Definitions and standard results} 

Let $G$ and $M$ be as above. We view $G$ as acting on itself by the adjoint action. 
Given a bivector $P\in \Gamma(M,\bigwedge^2 TM)$, we let $P^\sharp : T^\ast M \to TM$ be defined by\footnote{Our convention is $\langle X_1\otimes X_2, \dd f_1 \otimes \dd f_2 \rangle:= X_1(f_1)\, X_2(f_2)$ for vector fields $X_1,X_2$ and functions $f_1,f_2$ on $M$.} $P^\sharp(\alpha_1)(\alpha_2)=\langle P, \alpha_1 \otimes \alpha_2 \rangle$ for $1$-forms $\alpha_{1,2}$ on $M$. 
\begin{definition} \label{def:HamqP}
A $G$-invariant bivector $P$ on $M$ is \emph{quasi-Poisson} if $[P,P]=\phi_M$, where $[-,-]$ denotes the Schouten-Nijenhuis bracket on $M$. Then $(M,P)$ is called a \emph{quasi-Poisson variety}. 

\noindent The quasi-Poisson bivector $P$ is \emph{non-degenerate} if, for any $x\in M$, the following map is surjective~:  
\begin{equation} \label{Eq:nondeg}
 T^\ast_x M\times \g\to T_xM, \quad (\alpha, \xi)\mapsto P^\sharp_x(\alpha)+\xi_{M,x}\,.
\end{equation}
 
\noindent The quasi-Poisson variety $(M,P)$ is \emph{Hamiltonian} if it admits a \emph{moment map}, that is a $G$-equivariant morphism $\Phi:M\to G$ satisfying for any function $F$ on $G$, 
\begin{equation} \label{Eq:momap}
  P^\sharp(\dd (\Phi^\ast F))= (\Phi^\ast \mathcal{D}(F))_M\,. 
\end{equation} 
\end{definition}

Given a $2$-form $\omega$ on $M$, we let $\omega^\flat:TM\to T^\ast M$ be defined by $\omega^\flat(X_1)(X_2) = \langle X_1\otimes X_2,\omega \rangle$.  
\begin{definition}  \label{def:qHam}
Fix a $G$-invariant $2$-form $\omega$ on $M$ and a $G$-equivariant morphism $\Phi:M \to G$. 
We say that the triple $(M,\omega,\Phi)$ is a \emph{quasi-Hamiltonian variety} (with \emph{moment map} given by $\Phi$)  if the following three conditions are satisfied:
\begin{subequations}
 \begin{align}
  &\dd\omega=\Phi^\ast \eta\,, \label{Eq:B1} \\
  &\omega^{\flat}(\xi_M)=\frac12 \Phi^\ast (\xi,\,\theta^L+\theta^R\,)_{\g} \,, \quad \forall \xi \in \g\,, \label{Eq:B2} \\
  &T_xM \to T^\ast_x M\times \g, \,\, X\mapsto (\omega_x^\flat(X)\,, \langle X,(\Phi^\ast \theta^L)_x\rangle ), \text{ is injective } \quad \forall x\in M\,. \label{Eq:B3}
 \end{align}
\end{subequations} 
\end{definition}

\begin{remark} \label{Rem:Conv-qH}
We follow \cite{AKSM,VdB1} for the rule \eqref{Eq:B1}, which is taken with an opposite sign in \cite{AMM,Ya}. 
The rule \eqref{Eq:B3} is stated in that form in \cite{VdB2}, which is equivalent to explicitly characterizing $\ker \omega_x$ as in \cite{AKSM,AMM}, see \cite[Lemma 4.1]{VdB2} or \cite[\S4.2]{Hu}.  
\end{remark}

\begin{proposition} \label{Pr:Corr}
 There is a $1$-$1$ correspondence between structures of non-degenerate Hamiltonian quasi-Poisson variety and those of quasi-Hamiltonian variety on $M$ admitting the (same) moment map $\Phi:M\to G$ for which the associated quasi-Poisson bivector $P$ and $2$-form $\omega$ are related by: 
\begin{equation} \label{Eq:corrPOm}
 P^\sharp \circ \omega^\flat = \Id_{TM}- \frac14 \sum_{a} (E_a)_M \otimes \Phi^\ast(E^a,\theta^L-\theta^R)_{\g}\,.
\end{equation}
\end{proposition}

\begin{lemma} \label{Lem:Compat}
Let $(M,P,\Phi)$ be a Hamiltonian quasi-Poisson variety. 
Assume that there exists a $2$-form $\omega$ such that $(M,\omega,\Phi)$ satisfies \eqref{Eq:B1}, \eqref{Eq:B2},
and the compatibility condition \eqref{Eq:corrPOm} holds for $P$ with $\omega$. 
Then $P$ is non-degenerate and $(M,\omega,\Phi)$ is a quasi-Hamiltonian variety. 
\end{lemma}

\begin{proposition}[Fusion] \label{Pr:Fus}
Let $(M,P,\Phi)$ be a Hamiltonian quasi-Poisson variety for an action of $G\times G\times H$. 
Write $\Phi=(\Phi_1,\Phi_2,\Phi_H)$ componentwise with $\Phi_H:M\to H$, and $\Phi_j:M\to G$ valued in the $j$-th factor of $G$ for $j=1,2$. 
Then, for the diagonal action of  $G\times H \hookrightarrow G\times G \times H$, $(g,h)\mapsto (g,g,h)$, on $M$, 
the triple $(M,P-\psi_{\mathrm{fus}},(\Phi_1\Phi_2,\Phi_H))$ is a Hamiltonian quasi-Poisson variety for   
\begin{equation}\label{Eq:Fus-Psi}
 \psi_{\mathrm{fus}}:=\frac12 \sum_{a} (E_a,0)_M\wedge (0,E^a)_M\,,  
\end{equation}
where we use dual bases $(E_a)_a$ and $(E^a)_a$ of $\g$ in the infinitesimal action of $\g\times \g$ on $M$. 

Furthermore, assume that $P$ is non-degenerate and denote by $\omega$ the quasi-Hamiltonian $2$-form corresponding to the action of $G\times G\times H$. Then  $P-\psi_{\mathrm{fus}}$ is non-degenerate, and $\omega- \omega_{\mathrm{fus}}$ is the quasi-Hamiltonian $2$-form corresponding to the action of $G\times H$ for 
\begin{equation} \label{Eq:Fus-omeg}
\omega_{\mathrm{fus}}:= \frac12 (\Phi_1^{-1} \dd\Phi_1 \stackrel{\wedge}{,} \dd\Phi_2\, \Phi_2^{-1})_\g\,.
\end{equation} 
\end{proposition}

\subsection{Multiplicative quiver variety}  \label{ss:MQV}

\subsubsection{Quiver representations}

Fix a quiver $Q$ and a dimension vector $\nfat=(n_s)\in \Z_{\geq 0}^{I}$. 
A representation of dimension $\nfat$ of the path algebra $\CC Q$ is an element  $\rho\in \Hom_B(\CC Q,\End(\CC^\nfat))$ such that, for each $s\in I$, $\rho(e_s)$ is the projection on the summand $\CC^{n_s}$ of $\CC^\nfat:=\oplus_{r\in I} \CC^{n_r}$.
Using the convention of \ref{ss:Not}, we see that the matrix $\rho(a)$, $a\in Q$, may only have a nonzero block of size $n_{t(a)}\times n_{h(a)}$ corresponding to a linear map\footnote{We warn the reader that we use Van den Bergh's convention  \cite{VdB1,VdB2} for quiver representations, which is the opposite of the convention taken in \cite{CBShaw,Ya}.} 
$\CC^{n_{h(a)}}\to \CC^{n_{t(a)}}$.
Therefore the space of all representations of dimension $\nfat=(n_s)\in \Z_{\geq 0}^{I}$, denoted $M_Q:=M_Q(\nfat)$, is an affine space $\mathbb{A}^d_\CC$ of dimension $d:=\sum_{a\in Q} n_{t(a)}n_{h(a)}$. This turns $M_Q$ into an affine variety.
Rather than working with points as representations of $\CC Q$, we view a point $\Xtt\in M_Q$ as being parametrized as $\Xtt=(\Xtt_a)_{a\in Q}$ with $\Xtt_a:=\rho(a)$ for $\rho$ the corresponding representation.

There is a natural left action of $\Gl_{\nfat}:=\prod_{s\in I}\Gl_{n_s}(\CC)$ on $M_Q$ given by the morphism
\begin{equation} \label{Eq:Act-Gln}
 \Gl_{\nfat}\times M_Q \to M_Q, \quad (g=(g_s)_{s\in I}, \Xtt) \mapsto g\cdot \Xtt:= (g_{t(a)}\Xtt_a g^{-1}_{h(a)})_{a\in Q}\,.
\end{equation}
This induces a left action on functions given by $(g\cdot f)(\Xtt)=f(g^{-1}\cdot \Xtt)$ for any function $f$ on $M_Q$.
Abusing notation, we introduce for each $a\in Q$ the morphism $\Xtt_a:M_Q\to \End(\CC^\nfat)$ returning the matrix representing $a\in Q$ at each point.
The $\Gl_{\nfat}$-action can then be written as $g\cdot \Xtt_a=g^{-1}\Xtt_a g$.
Differentiating this, we get an infinitesimal action of $\gl_{\nfat}:=\prod_{s\in I}\gl_{n_s}(\CC)$ on functions that satisfies $\xi_{M_Q}(\Xtt_a)=[\Xtt_a,\xi]$, cf. \eqref{EqinfVectM}.

\medskip

Take a stability parameter $\theta:=(\theta_s)\in \Q^I$ such that $\theta\cdot \nfat:=\sum_{s\in I} \theta_s n_s = 0$. 
We say that $\Xtt\in M_Q$ is $\theta$-semistable (resp. $\theta$-stable) if $\theta\cdot \dim(V)\leq 0$ (resp. $\theta\cdot \dim(V)< 0$) for any nonempty vector subspace $V=\oplus_{s\in I} V_s \subsetneq \CC^\nfat$ such that $\Xtt_a(V_{h(a)})\subset V_{t(a)}$. We let 
\begin{equation}
\begin{aligned}
  M_Q^{\theta.s}:=\{\Xtt \in M_Q \mid \Xtt \text{ is } \theta\text{-stable}\}\,, \qquad
   M_Q^{\theta.ss}:=\{\Xtt \in M_Q \mid \Xtt \text{ is } \theta\text{-semistable}\}\,. 
\end{aligned}
\end{equation}

\subsubsection{Definition} 

Introduce the smooth open affine subvariety 
$M_{\overline{Q}}^\circ:=\{\Xtt \mid \det (\Id+\Xtt_a \Xtt_{a^\ast}) \neq 0 \, \forall a \in \overline{Q}\}\subset M_{\overline{Q}}$. 
Fix a total ordering $<$ on $\overline{Q}$.
We reproduce from \eqref{Eq:Intro1} the $\Gl_{\nfat}$-equivariant morphism 
\begin{equation} \label{Eq:momapMQV}
\Phi:M_{\overline{Q}}^\circ\to \Gl_{\nfat},\quad 
\Phi(\Xtt):=\prod_{a \in \overline{Q}} (\Id+\Xtt_a \Xtt_{a^\ast})^{\epsilon(a)}\,,
\end{equation}
where factors are taken from left to right with respect to the ordering $<$ on the arrows.
Similarly, set
$\Phi_a:=\prod_{b<a} (\Id+\Xtt_b \Xtt_{b^\ast})^{\epsilon(b)}$ for all $a\in \overline{Q}$. 
Finally, fix $\qfat:=(q_s)\in (\CC^\times)^I$ and denote in the same way the central element $(q_s\Id_{n_s})_s\in \Gl_{\nfat}$.
Note that $\Phi^{-1}(\qfat)$ is a $\Gl_{\nfat}$-invariant closed subvariety of $M_{\overline{Q}}^\circ$, which is empty if $\qfat^\nfat:=\prod_{s\in I} q_s^{n_s}\neq 1$ because $\det \circ \,\Phi= 1$. 
This subvariety parametrizes representation of the multiplicative preprojective algebra $\Lambda^{\qfat}(Q)$, cf.~\cite{CBShaw}. 
We are in position to define the main objects at stake, see \cite[\S2.2]{Ya} for the relevant geometric invariant theory of quivers. 

\begin{definition}[\cite{CBShaw,Ya}] \label{Def:MQV}
The multiplicative quiver variety associated with $\qfat$ and stability parameter $\theta$ is the good quotient 
\begin{equation}
 \Mcal_{\overline{Q},\qfat,\theta}:=(M_{\overline{Q}}^{\theta.ss}\cap \Phi^{-1}(\qfat))/\!/ \Gl_{\nfat}\,.
\end{equation}
Its smooth locus is given by the geometric quotient $\Mcal_{\overline{Q},\qfat,\theta}^s:=(M_{\overline{Q}}^{\theta.s}\cap \Phi^{-1}(\qfat))/ \Gl_{\nfat}$. 
\end{definition}

\subsubsection{(Quasi-)Poisson structure and corresponding \texorpdfstring{$2$-}{2-}form}

Set $\partial_a\in \mathcal{X}^1(M_{\overline{Q}},\gl_\nfat)$ with $(i,j)$ entry given by the vector field   
$(\partial_a)_{ij}:=\partial/\partial (\Xtt_a)_{ji}$, i.e. 
\begin{equation*}
((\partial_a)_{ij}f)(\Xtt)=\left.\frac{d}{dt}\right|_{t=0} f((\Xtt_b+t \delta_{ab} E_{ji})_{b\in \overline{Q}}) \,,
\end{equation*}
for any function $f$ around $\Xtt\in M_{\overline{Q}}$. 
In particular, as a matrix $\partial_a$ only consists of a nonzero block of size $n_{h(a)}\times n_{t(a)}$, and $\langle (\partial_a)_{ij} , \dd(\Xtt_a)_{kl} \rangle=\delta_{jk} \delta_{il}$ with indices $i,j,k,l$ taken with respect to the nonzero blocks of $\partial_a$ and $\Xtt_a$. 
Note that the infinitesimal action of $\xi \in \gl_{\nfat}$ on $M_{\overline{Q}}$ can be written as the vector field
\begin{equation} \label{Eq:Act-inf}
 \xi_{M_{\overline{Q}}} = \sum_{a\in \overline{Q}}\tr((\partial_{a} \Xtt_{a} -  \Xtt_a \partial_a)\xi)
 =\sum_{a\in \overline{Q}}\tr((\partial_{a^\ast} \Xtt_{a^\ast} -  \Xtt_a \partial_a)\xi) \,.
\end{equation}

\begin{theorem}[\cite{VdB1,VdB2,Ya}]  \label{Thm:qPqHam-MQV}
The smooth complex variety $M_{\overline{Q}}^\circ$ is endowed with a structure of Hamiltonian quasi-Poisson variety for 
the moment map $\Phi$ \eqref{Eq:momapMQV} and the quasi-Poisson bivector 
\begin{equation} \label{Eq:qP-MQV}
\begin{aligned}
 P:=&\, \frac12 \sum_{a\in \overline{Q}} \epsilon(a)\, 
 \tr\left[(\Id+\Xtt_{a^\ast}\Xtt_a) \, \partial_a \wedge \partial_{a^\ast} \right]  \\
& -\frac12 \sum_{\substack{a<b \\ a,b\in \overline{Q}}} 
 \tr\left[(\partial_{a^\ast} \Xtt_{a^\ast} - \Xtt_{a}\, \partial_a) \wedge (\partial_{b^\ast} \Xtt_{b^\ast} - \Xtt_{b} \, \partial_b) \right] \,. 
\end{aligned}
\end{equation} 
Moreover, $P$ is non-degenerate, with corresponding $2$-form 
\begin{equation} \label{Eq:omega-MQV}
\begin{aligned}
  \omega:=&\, -\frac12 \sum_{a\in \overline{Q}} \epsilon(a)\, 
 \tr\left[(\Id+\Xtt_a \Xtt_{a^\ast})^{-1} \, \dd\Xtt_a \wedge \dd\Xtt_{a^\ast} \right]  \\
& -\frac12 \sum_{ a\in \overline{Q}} 
 \tr\left[ \Phi_a^{-1} \, \dd\Phi_a \wedge \dd(\Id+\Xtt_a \Xtt_{a^\ast})^{\epsilon(a)}\,\, (\Id+\Xtt_a \Xtt_{a^\ast})^{-\epsilon(a)}  \right] \,,
\end{aligned}
\end{equation} 
that makes $(M_{\overline{Q}}^\circ,\omega,\Phi)$ a quasi-Hamiltonian variety. 
The structures are independent of the choice of ordering or the orientation of arrows, up to isomorphism. 
\end{theorem}

\begin{corollary}[\cite{VdB1,VdB2,Ya}]  \label{Cor:qPqHam-redMQV}
The multiplicative quiver variety $\Mcal_{\overline{Q},\qfat,\theta}$, if not empty, is a Poisson variety with Poisson bivector given by \eqref{Eq:qP-MQV}. The Poisson structure is non-degenerate on $\Mcal_{\overline{Q},\qfat,\theta}^s$ where the corresponding symplectic form is given by \eqref{Eq:omega-MQV}. 
\end{corollary}


\begin{remark}
 Let us emphasize that Yamakawa \cite{Ya} uses a convention equivalent to writing paths from right to left in $\CC\overline{Q}$ and follows the definition of quasi-Hamiltonian space from \cite{AMM} (cf. Remark \ref{Rem:Conv-qH}). This explains the minor difference between \eqref{Eq:omega-MQV} and the $2$-form  in \cite[Prop. 3.2]{Ya}.
\end{remark}

\begin{remark} \label{Rem:MQV-fusion}
 Recall that Theorem \ref{Thm:qPqHam-MQV} is obtained by applying the method of fusion \cite{AKSM,AMM}. 
Therefore it suffices to prove the result for $M_{\overline{Q}}^\circ$ seen as a $\prod_{a\in \overline{Q}}\Gl_{n_{t(a)}}$-variety where for each $a\in \overline{Q}$, the action of the corresponding subgroup is by 
$\Xtt \mapsto g_{(a)}\cdot \Xtt:=(g_{(a)} \Xtt_a, \Xtt_{a^\ast} g_{(a)}^{-1},\Xtt_b)_{b\neq a,a^\ast}$. 
In that case only the first sum in \eqref{Eq:qP-MQV} and \eqref{Eq:omega-MQV} is considered, while the moment map consists of the morphisms  
$\Phi_{(a)}:M_{\overline{Q}}^\circ \to \Gl_{n_{t(a)}}$, $\Xtt\mapsto (\Id+ \Xtt_a \Xtt_{a^\ast})^{\epsilon(a)}$. 
This amounts to proving the result for a totally separated quiver, cf. \cite[\S6.7]{VdB1} and \cite{VdB2,Ya}. 
\end{remark}


\section{The (quasi-)Poisson pencil and its variants} \label{S:Pencil}

We introduce the notion of a quasi-Poisson pencil and derive some elementary properties. 
We construct such pencils in the presence of $\CC^\times$-actions and we apply this formalism to the open representation spaces $M_{\overline{Q}}^\circ$ that lead to multiplicative quiver varieties, as well as related spaces.  

\subsection{General framework} \label{s:PenGen}

Recall that on a variety $M$, a family of Poisson bivectors $(P_j)_{j\in J}$ forms a \emph{Poisson pencil} if the linear combination $\sum_{j\in J_0} c_j P_j$ is a Poisson bivector for any finite subset $J_0\subset J$ and  $(c_j)_j\in \CC^{J_0}$. 
The pencil has \emph{order} $r$ if $r\geq 1$ is the largest integer for which there exist $j_1,\ldots,j_r\in J$ and a point $x\in M$ such that $P_{j_1,x},\ldots,P_{j_r,x} \in \bigwedge^2 T_xM$ are independent. 
Pencils of order $r=2$ play a central role in the bihamiltonian approach to integrability \cite{KSM}.  

\begin{definition} \label{def:Pencil}
Assume that $M$ is a $G$-variety and let $P_0$ be a quasi-Poisson bivector on $M$. 
A family of bivectors $(P_j)_{j\in J}$ forms a \emph{quasi-Poisson pencil centered at} $P_0$ if, for any finite subset $J_0\subset J$ and $(c_j)_j\in \CC^{J_0}$, the linear combination $P_0+\sum_{j\in J_0} c_j P_j$ is also a quasi-Poisson bivector. 
The pencil has order $r$ if $r\geq 1$ is the largest integer for which there exist $j_1,\ldots,j_r\in J$ and a point $x\in M$ such that\footnote{As opposed to \cite{FF23}, we do not require the $r$ bivectors to be independent from $P_{0,x}$ at $x\in M$. This last independence condition may fail in some cases, although it is automatic when the induced trivector $\phi_{M,x}$ is nonzero due to part 1 of Lemma \ref{Lem:AltPenc}.} $P_{j_1,x},\ldots,P_{j_r,x} \in \bigwedge^2 T_xM$ are independent. 

\noindent If there exists a morphism $\Phi:M\to G$ which is a moment map for any quasi-Poisson bivector $P_0+\sum_{j\in J_0} c_j P_j$ as above, we say that the pencil is \emph{Hamiltonian with moment map} $\Phi$. 
\end{definition}

\subsubsection{First properties} 
Without loss of generality, we assume that the index set $J$ is finite. 
The bivector fields in a quasi-Poisson pencil are far from being arbitrary. 
\begin{lemma} \label{Lem:AltPenc}
For a $G$-variety $M$, 
Let $(P_j)_{j\in J}$ be a finite family of bivectors on $M$, and $P_0$ be a quasi-Poisson bivector on $M$. 
Then:
\begin{enumerate}
 \item $(P_j)_{j\in J}$ defines a quasi-Poisson pencil of order $r$ centered at $P_0$ if and only if 
 $(P_j)_{j\in J}$ defines a $G$-invariant Poisson pencil of order $r$ and, for any $j\in J$, $P_j$  satisfies $[P_0,P_j]=0$. 
\item When $P_0$ is Hamiltonian for the moment map $\Phi:M\to G$, the quasi-Poisson pencil is Hamiltonian for $\Phi$ if and only if $P_j^\sharp(\dd(\Phi^\ast F))=0$ for any $j\in J$ and function $F$ on $G$. 
\end{enumerate}
\end{lemma}
\begin{proof}
For part (1), assume that the pencil is quasi-Poisson. Given parameters $\zbar:=(z_j)_j\in \CC^J$, 
write $P(\zbar)=\sum_{j\in J} z_j P_j$. 
Taking the Schouten bracket of the bivector $P(\zbar)$ with itself yields 
\begin{equation} \label{Eq:Pf-penc}
[P(\zbar),P(\zbar)]=[P_0+P(\zbar),P_0+P(\zbar)] - 2 [P_0,P(\zbar)] - [P_0,P_0] = - 2 [P_0,P(\zbar)]\,, 
\end{equation}
since $P_0$ and $P_0+P(\zbar)$ are quasi-Poisson. 
Fix $j\in J$, and take $z_k=0$ for all $k\neq j$ so that \eqref{Eq:Pf-penc} yields 
$z_j^2 [P_j,P_j]=-2z_j [P_0,P_j]$; varying $z_j\in \CC$ implies that $[P_0,P_j]=0$. Thus  \eqref{Eq:Pf-penc} is identically zero and each bivector $P(\zbar)$ is Poisson. 
By definition, quasi-Poisson bivectors are $G$-invariant, so that $P_j=(P_0+P_j)-P_0$ is itself $G$-invariant for any $j\in J$.  
The fact that the order is the same is direct from the definition.  
The converse is proved similarly.

For part (2), assume that $\Phi$ is a moment map for the whole pencil. Then for any $j\in J$,
\begin{equation} \label{Eq:Pf-penc-2}
P_j^\sharp(\dd(\Phi^\ast F))=(P_0+P_j)^\sharp(\dd(\Phi^\ast F)) - P_0^\sharp(\dd(\Phi^\ast F))=0\,, 
\end{equation}
as both terms cancel out since they satisfy \eqref{Eq:momap}. The converse is proved similarly.  
\end{proof}

Fusion as in Proposition \ref{Pr:Fus} is compatible with pencils in the following way. 
\begin{lemma}  \label{Lem:FusPenc}
 Let $(M,P_0,\Phi)$ be a Hamiltonian quasi-Poisson variety for an action of $G\times G\times H$. 
Write $\Phi=(\Phi_1,\Phi_2,\Phi_H)$ componentwise with $\Phi_H:M\to H$ and $\Phi_i:M\to G$ valued in the $i$-th factor of $G$ for $i=1,2$. Then:
\begin{enumerate}
 \item Assume that $(P_j)_{j\in J}$ defines a Hamiltonian quasi-Poisson pencil centered at $P_0$ with moment map $\Phi$. 
 Then, after fusion, it defines a Hamiltonian quasi-Poisson pencil centered at $P_0-\psi_{\mathrm{fus}}$ with moment map $(\Phi_1\Phi_2,\Phi_H)$, where $\psi_{\mathrm{fus}}$ is given by \eqref{Eq:Fus-Psi}.  
 \item For $\zbar:=(z_j)_j\in \CC^J$, assume that before fusion $Q(\zbar):=P_0+\sum_j z_j P_j$ is non-degenerate with corresponding quasi-Hamiltonian $2$-form $\omega(\zbar)$. 
 Then $Q(\zbar)-\psi_{\mathrm{fus}}$ is non-degenerate with corresponding  quasi-Hamiltonian $2$-form $\omega(\zbar)-\omega_{\mathrm{fus}}$ where $\omega_{\mathrm{fus}}$ is given by \eqref{Eq:Fus-omeg}. 
 \item The order of the quasi-Poisson pencil stays the same after fusion. 
\end{enumerate}
\end{lemma}
\begin{proof}
For fixed $\zbar:=(z_j)_j\in \CC^J$ and $Q(\zbar):=P_0+\sum_j z_j P_j$, any $Q(\zbar)-\psi_{\mathrm{fus}}$ is a Hamiltonian quasi-Poisson bivector for the moment map $(\Phi_1\Phi_2,\Phi_H)$ due to Proposition \ref{Pr:Fus}. Part (1) directly follows. 
For part (2), we make the same reasoning with the $2$-form and apply  Proposition \ref{Pr:Fus} again. 
For part (3), it suffices to use the alternative characterization of a quasi-Poisson pencil given by item (1) in Lemma \ref{Lem:AltPenc}. 
\end{proof}

Finally, we make some observations regarding quasi-Poisson reduction. 
Given a quasi-Poisson pencil as in Definition \ref{def:Pencil},  the family $(P_0,P_j)_{j\in J}$ defines a Poisson pencil on $M/\!\!/G$ of order\footnote{By abusing terminology, we allow the order to be zero; this occurs e.g. if $M/\!\!/G$ is $0$-dimensional.} $r'\leq r+1$. 
In the Hamiltonian case, the family $(P_0,P_j)_{j\in J}$ defines a Poisson pencil of order $r'\leq r+1$ on any reduction $\Phi^{-1}(\mathcal{C})/\!\!/G$ (where $\mathcal{C}$ is the closure of a conjugacy class in $G$); this is also true in the presence of a stability parameter.

\subsubsection{Construction using an abelian group action}

We start by adapting to the algebraic setting an observation for constructing compatible quasi-Poisson bivectors made in \cite[\S3.2]{FF23}. 

Let $M$ be a variety with an action of $G$ as in \ref{ss:Setup}.
Assume that $(M,P,\Phi)$ is a Hamiltonian quasi-Poisson variety. 

\begin{lemma}   \label{Lem:Pencil-1}
Let $H$ be an abelian reductive algebraic group and $\h$ be its Lie algebra. 
Assume that $H$ acts on $M$ such that both $P$ and $\Phi$ are $H$-invariant. 
Given a basis $\chi^{(1)},\ldots,\chi^{(\ell)}$ of $\h$, construct the following bivector on $M$ from their infinitesimal vector fields:
\begin{equation}
 \psi_{\zbar}=\sum_{1\leq i<j\leq \ell} z_{ij}\, \chi^{(i)}_M \wedge \chi^{(j)}_M\,,
\end{equation}
where $\zbar=(z_{ij})_{i<j}\in \CC^{\ell(\ell-1)/2}$. 
Then $(M,P+\psi_{\zbar},\Phi)$ is a Hamiltonian quasi-Poisson $G$-variety. 
\end{lemma}
\begin{proof}
 Since $P$ is $H$-invariant and $H$ is abelian, we obtain $[\chi_M,P]=0$ and $[\chi_M,\chi'_M]=0$ for any $\chi,\chi'\in \h$. Thus $[P+\psi_{\zbar},P+\psi_{\zbar}]=[P,P]$ and the bivector is quasi-Poisson. 
 
 By $H$-invariance of $\Phi$, $\langle \chi_M , \dd\Phi^\ast F\rangle=0$ for any function $F$ on $G$ and $\chi \in \h$.  This implies that \eqref{Eq:momap} is satisfied with $P+\psi_{\zbar}$ and $\Phi$. 
\end{proof}

Now, we work with $P$ which is non-degenerate, and there exists a corresponding quasi-Hamiltonian structure whose $2$-form is denoted $\omega$.

\begin{proposition}   \label{Pr:Pencil}
Under the assumptions of Lemma \ref{Lem:Pencil-1}, assume that there exists a closed invariant $2$-form $\varpi_{\zbar}$ such that $\varpi_{\zbar}^\flat(\xi_M)=0$ for all $\xi\in \g$, and the compatibility condition \eqref{Eq:corrPOm} holds for $P+\psi_{\zbar}$ with $\omega+\varpi_{\zbar}$.
Then $(M,P+\psi_{\zbar},\Phi)$ is a non-degenerate Hamiltonian quasi-Poisson $G$-variety and 
$(M,\omega+\varpi_{\zbar},\Phi)$ is the corresponding quasi-Hamiltonian $G$-variety. 
\end{proposition}
\begin{proof}
 By Lemma \ref{Lem:Pencil-1}, $(M,P+\psi_{\zbar},\Phi)$ is a Hamiltonian quasi-Poisson $G$-variety. 
 By the assumptions on $\varpi_{\zbar}$, the triple $(M,\omega+\varpi_{\zbar},\Phi)$ satisfies \eqref{Eq:B1} and \eqref{Eq:B2}. 
The compatibility \eqref{Eq:corrPOm} allows to conclude from Lemma \ref{Lem:Compat}. 
\end{proof}

In the situations presented above, we obtain a Hamiltonian quasi-Poisson pencil centered at $P$ with moment map $\Phi$, according to Definition \ref{def:Pencil}. 
If the infinitesimal vector fields $\chi^{(1)}_M,\ldots,\chi^{(\ell)}_M$ are independent, this pencil has order $r=\ell(\ell-1)/2$. 

\subsubsection{The trick: embeddings of $\CC^\times$ in the center of $G$}

We can construct the $2$-form $\varpi_{\zbar}$ that appears in Proposition \ref{Pr:Pencil} in the presence of some $\CC^\times$ actions obtained from the center of $G$.

\begin{proposition}   \label{Pr:Pencil-Cact}
Assume that $(M,P,\Phi)$ is a Hamiltonian quasi-Poisson $G$-variety, where $G$ admits a decomposition 
$G=G'\times \times_{j=1}^\ell G_j$ such that for each $1\leq j \leq \ell$, 
$\CC^\times$ embeds in the center of $G_j$ as diagonal matrices 
(i.e. there is a map $\lambda \mapsto \lambda \Id_{G_j}$). 
For each $1\leq j \leq \ell$, denote by $\Inf^{(j)}$ the infinitesimal action of $1\in \CC$ associated with the embedding of $\CC^\times$ into $G_j$.
Write the moment map as $\Phi=(\Phi',\Phi_1,\ldots,\Phi_\ell)$ with $\Phi':M\to G'$ and $\Phi_j:M\to G_j$, $1\leq j\leq \ell$. 
Furthermore, fix $\zbar=(z_{ij})_{1\leq i<j\leq \ell} \in \CC^{\ell(\ell-1)/2}$. Then: 
\begin{enumerate}
 \item The triple $(M,P+\psi_{\zbar},\Phi)$ is a Hamiltonian quasi-Poisson $G$-variety for 
\begin{equation} \label{Eq:PrPenc-psi}
 \psi_{\zbar}:= \sum_{1\leq i<j\leq \ell} z_{ij}\, \Inf^{(i)} \wedge \Inf^{(j)}\,.
\end{equation}
 \item Assume that $P$ is non-degenerate and $\omega$ is the corresponding $2$-form. Let 
\begin{equation} \label{Eq:PrPenc-varpi}
 \varpi_{\zbar}:= \sum_{1\leq i<j\leq \ell} z_{ij}\, \tr(\Phi_i^{-1} \,\dd\Phi_i) \wedge \tr(\Phi_j^{-1} \,\dd\Phi_j)\,.
\end{equation}
 Then $P+\psi_{\zbar}$ is non-degenerate, and $(M,\omega+\varpi_{\zbar},\Phi)$ defines the corresponding quasi-Hamiltonian variety. 
\end{enumerate}
\end{proposition}
\begin{proof}
 Part (1) directly follows from Lemma \ref{Lem:Pencil-1}. 
 
For part (2), note that $\varpi_{\zbar}$ is closed and $\varpi_{\zbar}^\flat(\xi_M)=0$ for any $\xi\in \g$ as 
$\tr(\Phi_j \,\dd\Phi_j)$ is $G$-invariant for $1\leq j \leq k$. 
Thus the statement holds by Proposition \ref{Pr:Pencil} if we can show the compatibility 
\begin{equation*}
 (P+\psi_{\zbar})^\sharp \circ (\omega+\varpi_{\zbar})^\flat = \Id_{TM}- \frac14 \sum_{a} (E_a)_M \otimes \Phi^\ast(E^a,\theta^L-\theta^R)_{\g} \,.
\end{equation*}
By assumption this holds for $\zbar=0$. Since $\psi_{\zbar}^\sharp \circ \varpi_{\zbar}^\flat=0$ by invariance of $\tr(\Phi_j \,\dd\Phi_j)$, we are left to prove that for any $X\in \g$, 
\begin{equation} \label{Eq:Pencil-Cact1}
 \psi_{\zbar}^\sharp \circ \omega^\flat(X) = - P^\sharp \circ \varpi_{\zbar}^\flat(X) \,.
\end{equation}
Using \eqref{Eq:PrPenc-psi}, the left-hand side of \eqref{Eq:Pencil-Cact1} reads 
\begin{equation*}
 \sum_{i<j} z_{ij}\, \big(\omega^\flat(\Inf^{(j)})(X)\,\, \Inf^{(i)} - \omega^\flat(\Inf^{(i)})(X)\,\, \Inf^{(j)}\big)\,.
\end{equation*}
Write $1_{\g_j}$ for the image of $1\in \CC$ obtained by differentiating the embedding $\CC^\times \hookrightarrow G_j$. 
By definition of the vector field $\Inf^{(j)}$ and \eqref{Eq:B2}, we get 
\begin{equation}
 \omega^\flat(\Inf^{(j)})=\frac12 (1_{\g_j},\Phi^{-1}\,\dd\Phi+\dd\Phi\, \Phi^{-1})_\g
 =\tr(\Phi_j^{-1} \,\dd\Phi_j)\,,
\end{equation}
since $\Phi_j$ is the diagonal block of $\Phi$ corresponding to its $G_j$-component. Hence 
\begin{equation}  \label{Eq:Pencil-Cact2}
\eqref{Eq:Pencil-Cact1}_{LHS}=
\sum_{i<j} z_{ij}\, ( \langle X, \tr(\Phi_j^{-1} \,\dd\Phi_j)\rangle \, \Inf^{(i)} - \langle X, \tr(\Phi_i^{-1} \,\dd\Phi_i)\rangle\, \Inf^{(j)}) \,.
\end{equation}
Meanwhile, we use \eqref{Eq:PrPenc-varpi} to expand the right-hand side of \eqref{Eq:Pencil-Cact1} as
\begin{equation} \label{Eq:Pencil-Cact3}
 - \sum_{i<j} z_{ij}\, \Big( \langle X, \tr(\Phi_i^{-1} \,\dd\Phi_i)\rangle\, P^\sharp(\tr(\Phi_j^{-1} \,\dd\Phi_j)) 
 - \langle X, \tr(\Phi_j^{-1} \,\dd\Phi_j)\rangle \, P^\sharp(\tr(\Phi_i^{-1} \,\dd\Phi_i)) \Big)\,.
\end{equation}
We can write $\tr(\Phi_j^{-1} \,\dd\Phi_j)=\sum_{u,v} (\Phi_j^{-1})_{vu} \,\dd(\Phi_j)_{uv}$, and let $g_{uv}:z\mapsto z_{uv}$ be the function on $G_j$ returning the $(u,v)$ entry of an element (recall that $G_j$ embeds in $G$ which is itself embedded in some $\Gl_N$). 
Note from \eqref{Eq:momap} that 
\begin{equation}
 P^\sharp(\dd(\Phi_j)_{uv}) = (\Phi_j^\ast \mathcal{D}(g_{uv}))_M
 =\frac12 \sum_a (E_a \Phi_j + \Phi_j E_a)_{uv}\, (E^a)_M\,, 
\end{equation}
where $(E_a)_a$, $(E^a)_a$ are dual bases of $\g_j$. This implies 
\begin{equation*}
P^\sharp(\tr(\Phi_j^{-1} \,\dd\Phi_j)) = (\sum_a \tr(E_a)\, E^a)_M = (1_{\g_j})_M = \Inf^{(j)}\,. 
\end{equation*}
Plugging this expression in \eqref{Eq:Pencil-Cact3} yields that the right-hand side of \eqref{Eq:Pencil-Cact1} has the same expansion as the left-hand side in \eqref{Eq:Pencil-Cact2}, as desired. 
\end{proof}

When the $\CC^\times$ actions considered in the statement of Proposition \ref{Pr:Pencil-Cact} define independent infinitesimal vector fields, this result yields a quasi-Poisson pencil centered at $P$ of order $\ell(\ell-1)/2$.  

\begin{remark} \label{Rem:Cstar-Fus}
In Proposition \ref{Pr:Pencil-Cact}, if some factors $G_j$ of $G$ are simply given by $\CC^\times$, then the corresponding terms appearing in $\psi_{\zbar}$ \eqref{Eq:PrPenc-psi} and $\varpi_{\zbar}$ \eqref{Eq:PrPenc-varpi} are multiple of the terms obtained by fusion of the corresponding actions. 
To see this, assume e.g. that $G_1=G_2=\CC^\times$ with $\{1\}$ being trivially an orthonormal basis of $\g_i=\CC$, $i=1,2$, then note that $\frac12 \Inf^{(1)} \wedge  \Inf^{(2)}$ is just $\psi_{\mathrm{fus}}$ \eqref{Eq:Fus-Psi}, and $\frac12 \Phi_1^{-1} \dd\Phi_1 \wedge  \Phi_2^{-1} \dd\Phi_2$ equals $\omega_{\mathrm{fus}}$ \eqref{Eq:Fus-omeg}. 
\end{remark}


\begin{remark}
In the real case with a compact Lie group $G$, an analogue of Proposition \ref{Pr:Pencil-Cact} consists in considering the $\operatorname{U}(1)$-actions coming from (the center of) factors $\operatorname{U}(n)$ in $G$. 
The pencil exhibited in \cite{FF23} can be derived as a special instance of this construction.
\end{remark}

\subsection{Case of multiplicative quiver varieties} \label{ss:PencilMQV}

Let $Q$ be a quiver and $\nfat$ a dimension vector. 
We continue with the notation of \ref{ss:MQV}.

For each $b\in \overline{Q}$,
there is a natural left action of $\CC^\times$ on $M_{\overline{Q}}$ given by 
\begin{equation} \label{Eq:Act-Ct}
\mathcal{A}_b : \CC^\times \times M_{\overline{Q}} \to M_{\overline{Q}}, \quad
(\lambda, \Xtt) \mapsto \lambda \cdot_b \Xtt:= (\lambda^{\delta_{ab} \epsilon(a)}\,\Xtt_a)_{a\in \overline{Q}}\,.
\end{equation}
This induces a left action on functions characterized by $\lambda \cdot_b \Xtt_b = \lambda^{-1} \Xtt_b$, $\lambda \cdot_b \Xtt_{b^\ast} = \lambda \Xtt_{b^\ast}$, and $\lambda \cdot_b \Xtt_a = \Xtt_a$ for $a\in \overline{Q}\setminus\{b,b^\ast\}$.
We obtain a corresponding infinitesimal action of $\CC$ on functions that satisfies
$z \cdot_b \Xtt_b = - z \Xtt_b$, $z \cdot_b \Xtt_{b^\ast} = z \Xtt_{b^\ast}$, and $z \cdot_b \Xtt_a = 0$ otherwise for $z\in \CC$.
We write $\mathcal{A}_b(z):=z \cdot_b (-)$ for the corresponding vector field on $M_{\overline{Q}}$. 
The following identity is straightforward:
\begin{equation} \label{Eq:Act-Ct-inf}
 \mathcal{A}_b(z) = z\, \tr(\Xtt_{b^\ast}\,\partial_{b^\ast}  - \Xtt_{b}\, \partial_b)\,.
\end{equation} 
The vector field $\mathcal{A}_b$ restricts to $M_{\overline{Q}}^\circ$ as  the moment map $\Phi$ \eqref{Eq:momapMQV} is invariant under \eqref{Eq:Act-Ct}. 

To state the next result, recall $P,\omega$ and $\Phi$ given by \eqref{Eq:qP-MQV}, \eqref{Eq:omega-MQV} and \eqref{Eq:momapMQV} that define on $M_{\overline{Q}}^\circ$ the structures, presented in Theorem \ref{Thm:qPqHam-MQV}, of non-degenerate Hamiltonian quasi-Poisson and quasi-Hamiltonian variety which correspond to one another.

\begin{theorem}   \label{Thm:Pencil-MQV}
 Fix $\zbar=(z_{a,b})_{a<b}$ where $z_{a,b}\in \CC$ for each $a,b\in Q$ with $a<b$. 
Define the following bivector  and $2$-form on $M_{\overline{Q}}^\circ$: 
\begin{align}
 \psi_{\zbar}&= \sum_{a<b} z_{a,b}\,  \tr \left[\Xtt_{a^\ast}\,\partial_{a^\ast}  - \Xtt_{a}\, \partial_a \right] \wedge 
 \tr \left[\Xtt_{b^\ast}\,\partial_{b^\ast}  - \Xtt_{b}\, \partial_b \right]\,, \label{Eq:psiMQV} \\
 \varpi_{\zbar}&= \sum_{a<b} z_{a,b}\,  \tr \left[(\Id+\Xtt_a \Xtt_{a^\ast})^{-1} \, \dd(\Id+\Xtt_a \Xtt_{a^\ast}) \right] \wedge 
 \tr \left[(\Id+\Xtt_b \Xtt_{b^\ast})^{-1} \, \dd(\Id+\Xtt_b \Xtt_{b^\ast}) \right]\,. \label{Eq:varpiMQV}
\end{align}
Then: 
\begin{enumerate}
 \item the triple $(M_{\overline{Q}}^\circ, P+\psi_{\zbar},\Phi)$ is a non-degenerate Hamiltonian quasi-Poisson variety; 
 \item the triple $(M_{\overline{Q}}^\circ, \omega+\varpi_{\zbar},\Phi)$ is a quasi-Hamiltonian variety;
 \item the two triples correspond to one another via \eqref{Eq:corrPOm}. 
\end{enumerate}
\end{theorem} 
\begin{proof}
Note that the action $\mathcal{A}_b$ \eqref{Eq:Act-Ct} of $\CC^\times$ on $M_{\overline{Q}}^\circ$ is obtained using the embedding $\CC^\times \hookrightarrow \Gl_{n_{t(b)}}$, $\lambda \mapsto \lambda \Id_{n_{t(b)}}$, from the corresponding action considered in Remark \ref{Rem:MQV-fusion} \emph{before} fusion. 
Therefore it suffices to apply Proposition \ref{Pr:Pencil-Cact} to  $M_{\overline{Q}}^\circ$ before fusion, then perform fusion as in the original works of Van den Bergh and Yamakawa \cite{VdB1,VdB2,Ya}; we conclude by Lemma \ref{Lem:AltPenc}.  
\end{proof}

The range of parameters $\zbar$ in Theorem \ref{Thm:Pencil-MQV} is restricted to all $a<b$ with $a,b\in Q$ instead of $a,b\in \overline{Q}$ to avoid redundancies. 
Indeed, it is clear that $\mathcal{A}_b(z)=-\mathcal{A}_{b^\ast}(z)$ for each $b\in \overline{Q}$ and $z\in \CC$, and similarly using the notation from Remark \ref{Rem:MQV-fusion}, 
$\tr[\Phi_{(b)}^{-1} \dd\Phi_{(b)}]=-\tr[\Phi_{(b^\ast)}^{-1} \dd \Phi_{(b^\ast)}]$ for each $b\in \overline{Q}$. 
The infinitesimal actions $\mathcal{A}_b(1)$ \eqref{Eq:Act-Ct-inf} taken with $b\in Q$ are independent whenever each $X_b$ is not just a point, and we can conclude that Theorem \ref{Thm:Main1} holds.
We can in fact state the following generalization of Corollary \ref{Cor:qPqHam-redMQV}.

\begin{corollary}  \label{Cor:Pencil-redMQV}
The multiplicative quiver variety $\Mcal_{\overline{Q},\qfat,\theta}$, if not empty, is a variety equipped with a Poisson pencil of order $r\leq r_Q+1$, where $r_Q:=|Q|(|Q|-1)/2$.
For any $z_0\in \CC^\times$ and $\zbar\in \CC^{r_Q}$,
the Poisson bivector $z_0 P+\psi_{z_0\zbar}$ (cf. \eqref{Eq:qP-MQV} and \eqref{Eq:psiMQV}) is non-degenerate on $\Mcal_{\overline{Q},\qfat,\theta}^s$ where the corresponding symplectic form is given by $z_0^{-1} \omega + \varpi_{z_0^{-1}\zbar}$  (cf. \eqref{Eq:omega-MQV} and \eqref{Eq:varpiMQV}).
\end{corollary}

\begin{remark}
 Quantizations of multiplicative quiver varieties were considered in \cite{GJS,Jo} based on Van den Bergh's bivector $P$ \eqref{Eq:qP-MQV}. It would be interesting to understand how to generalize these works to encompass the whole pencil of Corollary \ref{Cor:Pencil-redMQV}. 
\end{remark}

Computing the order $r$ of the pencil inherited by $\Mcal_{\overline{Q},\qfat,\theta}$ appears to be a challenging problem. 
If the quiver $Q$ contains loops or multiple arrows between two vertices, 
we expect that $r>1$ (assuming that $\dim \Mcal_{\overline{Q},\qfat,\theta}>0$). We refer to Corollary \ref{Cor:RS-rank} with $d>2$ for a family of quivers where this holds. 
However, the added bivectors $\psi_{\zbar}$ \eqref{Eq:psiMQV} may vanish after quasi-Poisson reduction, leading to trivial cases; this occurs for quivers of $A_l$ type, for example. 
Let us see this in greater generalities for a star-shaped quiver. Following e.g. \cite[\S4]{Ya}, given $k\geq 1$ and $l:=(l_1,\ldots,l_k)\in \Z_{\geq 1}^k$, the quiver $Q_{k;l}$ corresponding to these data has for vertex set 
$I=\{0\}\cup \{(k',l') \mid 1\leq l' \leq l_{k'},\,\, 1\leq k' \leq k\}$ and exactly one arrow 
$(k',1)\to 0$ for each $1\leq k' \leq k$, and one arrow $(k',l')\to (k',l'-1)$ for each $2\leq l' \leq l_{k'}$ and $k'$. 
We call a quiver \emph{star-shaped} if it is of the form $Q_{k;l}$ (for some $k,l$), up to changing the orientation of some arrows. 

\begin{proposition}
 Assume that $Q$ is star-shaped and $\qfat\in (\CC^\times)^I$ is such that $\Mcal_{\overline{Q},\qfat,0}$ (with $\theta=0$) has positive dimension.   
Then the Poisson pencil constructed in Corollary \ref{Cor:Pencil-redMQV} has order $1$, i.e. it only consists of multiples of Van den Bergh's bivector $P$ \eqref{Eq:qP-MQV}. 
\end{proposition}
\begin{proof}
The coordinate ring $\CC[\Mcal_{\overline{Q},\qfat,0}]$ is generated by elements of the form $\tr(\gamma):=\tr(\Xtt_{a_1} \cdots \Xtt_{a_k})$, where $\gamma:=a_1 \cdots a_k \in \CC\overline{Q}$ is a closed path with $a_1,\ldots,a_k\in \overline{Q}$. (This is a consequence of the Le Bruyn-Procesi's theorem \cite{LP90}.) 
In the double of a star-shaped quiver, any closed path $\gamma$ contains as many factors of $a$ and $a^\ast$ for each fixed $a\in Q$, so that the associated element $\tr(\gamma)$ is invariant under each $\CC^\times$-action $\mathcal{A}_b$ \eqref{Eq:Act-Ct}, $b\in Q$. 
Thus, any infinitesimal action \eqref{Eq:Act-Ct-inf} acts trivially on $\CC[\Mcal_{\overline{Q},\qfat,0}]$, and the bivector $\psi_{\zbar}$ \eqref{Eq:psiMQV} is just zero. 
\end{proof}

\subsection{Variants of the pencil} \label{s:VarPencil}

We present several uses of Proposition \ref{Pr:Pencil-Cact} in situations that are analogous to the one of multiplicative quiver varieties. 

\subsubsection{Deformation of the moment map} \label{ss:DefoMomap}

Fix $\gamma:=(\gamma_a)_{a\in Q}\in \CC^{|Q|}$ and let $\gamma_{a}:=\gamma_{a^\ast}$ for any $a\in \overline{Q} \setminus Q$.
Write $M_{\overline{Q}}^{\gamma}:=\{\Xtt \mid \det (\gamma_a \Id+\Xtt_a \Xtt_{a^\ast}) \neq 0 \, \forall a \in \overline{Q}\}\subset M_{\overline{Q}}$.
This subvariety is empty if there exists $a$ such that $\gamma_a=0$ with $n_{t(a)}\neq n_{h(a)}$.
Using the approach of \cite[\S4.4]{Ya}, the following result is easily adapted from Theorem \ref{Thm:qPqHam-MQV} which corresponds to the case $\gamma=(1,\ldots,1)$.

\begin{theorem}  \label{Thm:qPqHam-MQV-bis}
The smooth complex variety  $M_{\overline{Q}}^{\gamma}$ is endowed with a structure of Hamiltonian quasi-Poisson variety for
the moment map
\begin{equation} \label{Eq:momapMQV-bis}
\Phi^\gamma:M_{\overline{Q}}^{\gamma}\to \Gl_{\nfat},\quad
\Phi^\gamma(\Xtt):=\prod_{a \in \overline{Q}} (\gamma_a \Id+\Xtt_a \Xtt_{a^\ast})^{\epsilon(a)}\,,
\end{equation}
and the quasi-Poisson bivector
\begin{equation} \label{Eq:qP-MQV-bis}
\begin{aligned}
 P^\gamma:=&\, \frac12 \sum_{a\in \overline{Q}} \epsilon(a)\,
 \tr\left[(\gamma_a \Id+\Xtt_{a^\ast}\Xtt_a) \, \partial_a \wedge \partial_{a^\ast} \right]  \\
& -\frac12 \sum_{\substack{a<b \\ a,b\in \overline{Q}}}
 \tr\left[(\partial_{a^\ast} \Xtt_{a^\ast} - \Xtt_{a}\, \partial_a) \wedge (\partial_{b^\ast} \Xtt_{b^\ast} - \Xtt_{b} \, \partial_b) \right] \,.
\end{aligned}
\end{equation}
The quasi-Poisson structure is non-degenerate, and the corresponding quasi-Hamiltonian structure is defined by the $2$-form
\begin{equation} \label{Eq:omega-MQV-bis}
\begin{aligned}
  \omega^\gamma:=&\, -\frac12 \sum_{a\in \overline{Q}} \epsilon(a)\,
 \tr\left[(\gamma_a \Id+\Xtt_a \Xtt_{a^\ast})^{-1} \, \dd\Xtt_a \wedge \dd\Xtt_{a^\ast} \right]  \\
& -\frac12 \sum_{ a\in \overline{Q}}
 \tr\left[ (\Phi_a^\gamma)^{-1} \, \dd\Phi_a^\gamma \wedge \dd(\gamma_a \Id+\Xtt_a \Xtt_{a^\ast})^{\epsilon(a)}\,\, (\gamma_a \Id+\Xtt_a \Xtt_{a^\ast})^{-\epsilon(a)}  \right] \,,
\end{aligned}
\end{equation}
where $\Phi_a^\gamma:=\prod_{b<a} (\gamma_b \Id+\Xtt_b \Xtt_{b^\ast})^{\epsilon(b)}$.
Furthermore, the structure is independent of the choice of ordering or the orientation of arrows, up to isomorphism.
\end{theorem}

While the correspondence \eqref{Eq:corrPOm} between the $2$ structures is known, 
we derive it in Appendix \ref{App:Corr}
as we are not aware of references presenting this computation\footnote{Let us nevertheless refer to \cite{BCS} for a proof in non-commutative geometry. We warn the reader that the reference follows Yamakawa's convention \cite{Ya} for writing paths from right to left.}. 
We arrive at the following generalization of Theorem \ref{Thm:Pencil-MQV}. 

\begin{theorem}   \label{Thm:Pencil-MQV-bis}
 Fix $\zbar=(z_{a,b})_{a<b}$ where $z_{a,b}\in \CC$ for each $a,b\in Q$ with $a<b$. 
Define on $M_{\overline{Q}}^\gamma$ the bivector $\psi_{\zbar}$ \eqref{Eq:psiMQV} and the $2$-form 
\begin{small}
 \begin{align} \label{Eq:varpiMQV-bis}
 \varpi_{\zbar}^\gamma:= \sum_{a<b} z_{a,b}\,  \tr \left[(\gamma_a \Id+\Xtt_a \Xtt_{a^\ast})^{-1} \, \dd(\gamma_a\Id+\Xtt_a \Xtt_{a^\ast}) \right] \wedge 
 \tr \left[(\gamma_b \Id+\Xtt_b \Xtt_{b^\ast})^{-1} \, \dd(\gamma_b \Id+\Xtt_b \Xtt_{b^\ast}) \right]. 
\end{align} 
\end{small}
Then the triple $(M_{\overline{Q}}^\gamma, P^\gamma+\psi_{\zbar},\Phi^\gamma)$ is a non-degenerate Hamiltonian quasi-Poisson variety, which corresponds to the quasi-Hamiltonian variety 
$(M_{\overline{Q}}^\circ, \omega^\gamma+\varpi_{\zbar}^\gamma,\Phi^\gamma)$.  
\end{theorem} 

By performing reduction, one can state the analogue of Corollary \ref{Cor:Pencil-redMQV} in the obvious way.

\subsubsection{Character varieties} \label{ss:CharVar}

Fix integers $r,g\geq 0$ such that $r+g>0$. For $n\geq1$, define 
$$M_{g,r,n}:=\{ (A_1,A_1^\ast,\ldots,A_g,A_g^\ast,Z_1,\ldots,Z_r) \mid A_i,A_i^\ast,Z_j \in \Gl_n(\CC) \}\,.$$ 
There is a natural action of $\Gl_n$ on $M_{g,r,n}$ by simultaneous conjugation of the $2g+r$ matrices $(A_i,A_i^\ast,Z_j)$. 
For $1\leq i\leq g$, we define as we did for quivers the matrix-valued vector field 
$\partial_{A_i}\in \mathcal{X}^1(M_{g,r,n},\Gl_n)$ with $(k,l)$ entry given by   
$(\partial_{A_i})_{kl}:=\partial/\partial (A_i)_{lk}$. 
We introduce in the exact same way the notations $\partial_{A_i^\ast}$ and $\partial_{Z_j}$, $1\leq j \leq r$. 
To state the next result, we let $(A_i^\ast)^\ast:= A_i$.   

\begin{theorem}  \label{Thm:qPqHam-CharVar}
The smooth complex variety  $M_{g,r,n}$ is endowed with a Hamiltonian quasi-Poisson pencil of order $g(g-1)/2$ centered at the quasi-Poisson bivector
\begin{small}
\begin{equation} \label{Eq:qP-CharVar}
\begin{aligned}
 P_{\mathrm{char}}:=&\, \frac12 \sum_{1\leq i\leq g} \,
 \tr\left[A_i^\ast A_i \, \partial_{A_i} \wedge \partial_{A_i^\ast} - A_i A_i^\ast \partial_{A_i^\ast} \wedge \partial_{A_i} \right] 
 +\frac12 \sum_{1\leq j \leq r}  
 \tr\left[ Z_j^2 \, \partial_{Z_j} \wedge \partial_{Z_j} \right] \\
& -\frac12 \sum_{1\leq i\leq g}
 \tr\left[(\partial_{A_i^\ast} A_i^\ast - A_i\, \partial_{A_i}) \wedge (\partial_{A_i} A_i - A_i^\ast \, \partial_{A_i^\ast}) \right] \\
& -\frac12 \sum_{1\leq i<k\leq g} \sum_{ \substack{C_i\in \{A_i,A_i^\ast\} \\ C_k\in \{A_k,A_k^\ast\} } }
 \tr\left[(\partial_{C_i^\ast} C_i^\ast - C_i\, \partial_{C_i}) \wedge (\partial_{C_k^\ast} C_k^\ast - C_k \, \partial_{C_k}) \right] \\
 & -\frac12  \sum_{ \substack{1\leq i\leq g \\ C_i\in \{A_i,A_i^\ast \} } }   \sum_{1\leq j\leq r}
 \tr\left[(\partial_{C_i^\ast} C_i^\ast - C_i\, \partial_{C_i}) \wedge 
 (\partial_{Z_j} Z_j - Z_j \, \partial_{Z_j}) \right]\\ 
 & -\frac12  \sum_{1\leq j<k \leq r}  
 \tr\left[(\partial_{Z_j} Z_j - Z_j \, \partial_{Z_j}) \wedge (\partial_{Z_k} Z_k - Z_k \, \partial_{Z_k}) \right]
 \,,
\end{aligned}
\end{equation}
\end{small}
with moment map  
\begin{equation} \label{Eq:momap-CharVar}
\Phi:M_{g,r,n}\to \Gl_{n},\quad
\Phi(A_i,A_i^\ast,Z_j):=\prod_{1\leq i\leq g} A_iA_i^\ast A_i^{-1} (A_i^\ast)^{-1} \,Z_1 \cdots Z_r\,. 
\end{equation}
A bivector from the pencil is given by $P_{\mathrm{char}}+\psi_{\zbar}$ for $\zbar=(z_{i,k})_{i<k}$, $z_{ik}\in \CC$, and  
\begin{equation}
  \psi_{\zbar}= \sum_{i<k} z_{i,k}\,  \tr \left[ A_i^\ast\,\partial_{A_i^\ast}  - A_i \, \partial_{A_i} \right] \wedge 
 \tr \left[ A_k^\ast \,\partial_{A_k^\ast}  - A_k \, \partial_{A_k} \right]\,.  \label{Eq:psiCharVar} 
\end{equation}
\end{theorem}
\begin{proof}
 If $r=0$, this is the case of a $g$-loop quiver considered in Theorem \ref{Thm:qPqHam-MQV-bis} where all deformation parameters $\gamma_a$ are set to zero. 
If $r>0$, we perform fusion with $r$ copies of $\Gl_n$ endowed with its Hamiltonian quasi-Poisson structure for the conjugation action and for which the moment map is the identity \cite[Prop.~3.1]{AKSM}.  
\end{proof}

\begin{remark}
 The extra terms for a bivector in the pencil do not depend on the elements $Z_1,\ldots,Z_r$, cf. \eqref{Eq:psiCharVar}. 
This is because the general construction of a pencil as in Proposition \ref{Pr:Pencil-Cact} depends on actions of $\CC^\times$, and in the present situation the center of $\Gl_n$ acts trivially under the actions $Z_j\mapsto g Z_j g^{-1}$ for $g\in \Gl_n$.   
\end{remark}

Character varieties are obtained by performing reduction of $M_{g,r,n}$ with respect to its $\Gl_n$-action. 
In general, the smooth loci of these varieties do not carry a symplectic form; for the Poisson structure to be generically non-degenerate, one needs to consider subvarieties of character varieties obtained by fixing the factors $Z_1,\ldots,Z_r$ to closures of conjugacy classes. 
One can prepare the passage to fixing conjugacy classes before reduction, and the corresponding subvariety of $M_{g,r,n}$ is of the form $M_{\overline{Q}}^{\gamma}$ considered in \ref{ss:DefoMomap} where $Q$ is a ``comet-shaped quiver'', an extension of a $g$-loop quiver by adding $r$ legs. We refer to \cite{CB13,ST,Ya} for more details on this construction. 
The upshot is that, for this last type of varieties, the pencil is made of non-degenerate bivectors and we can write the corresponding quasi-Hamiltonian structures as stated in Theorem \ref{Thm:qPqHam-MQV-bis}. 

\subsubsection{Generalized MQV}  \label{ss:GenMQV}

Following Boalch \cite{Bo15}, we consider quivers endowed with a color function, i.e. pairs $(Q,\mathfrak{c})$ where $\mathfrak{c}:Q\to C$ for some finite set $C$ is a map on the arrows. 
We assume that for each $c\in C$, the subquiver $Q_c=\mathfrak{c}^{-1}(c)$ (if not empty) is a multipartite graph: we can partition the vertices $I_c=\{s\in I \mid s=t(a) \text{ or }s=h(a) \text{ for some }a\in Q_c\}$ into disjoint subsets $I_{c,1},\ldots, I_{c,k_c}$ ($k_c \geq 2$) such that there is exactly one arrow $v_{rs}:r\to s$ for all $r\in I_{c,\ell}$ and $s\in I_{c,\ell'}$ with $1\leq \ell' < \ell \leq k_c$. 
(Any quiver can be seen as a colored quiver by taking $C=Q$ and $\mathfrak{c}=\id_Q$.)

Fix a pair  $(Q,\mathfrak{c})$ and $\nfat \in \Z^I_{\geq 0}$. 
We consider the double $\overline{Q}$ of $Q$ and $M_{\overline{Q}}:=M_{\overline{Q}}(\nfat)$ as in \ref{ss:MQV} (these are independent of $\mathfrak{c}$). 
At $\Xtt\in M_{\overline{Q}}$, we can form for each $c\in C$ the matrices 
\begin{equation}
 v_{c,+}:=\Id + \sum_{\substack{r\in I_{c,\ell},\,s\in I_{c,\ell'}\\ \text{with}\,\ell > \ell'}} \Xtt_{v_{rs}}\,, \quad 
 v_{c,-}:=\Id + \sum_{\substack{r\in I_{c,\ell},\,s\in I_{c,\ell'}\\ \text{with}\,\ell > \ell'}} \Xtt_{v_{rs}^\ast}\,,
\end{equation}
which, as elements of $\End(\oplus_{\ell=1}^{k_c} \oplus_{s\in I_{c,\ell}} \CC^{n_s})$, are regarded as an upper block-triangular and a lower block-triangular matrix, respectively. 
We then let 
\begin{equation}
 M_{\overline{Q}}^{\mathfrak{c},c}:= \{\Xtt \in M_{\overline{Q}} \mid v_{c,-}v_{c,+} \text{ admits an opposite Gauss decomposition}\} \subset  M_{\overline{Q}} \,.
\end{equation}
In other words, at a point of $M_{\overline{Q}}^{\mathfrak{c},c}$, we can factorize $v_{c,-}v_{c,+}=w_{c,+} \Phi_c w_{c,-}$ with $\Phi_c=(\Phi_{c,s})\in \oplus_{\ell=1}^{k_c} \oplus_{s\in I_{c,\ell}} \Gl_{n_s}$ which is block-diagonal, while  $w_{c,+}, w_{c,-}$ are unipotent upper block-triangular and lower block-triangular, respectively. 
The interested reader should consult \cite{Bo15} for precise definitions\footnote{One should bear in mind that Boalch's convention for writing paths and looking at representations are the opposite of those in the present text. A double quiver for us corresponds to a graph for Boalch. We refer to \cite{FFern} for the conventions that we follow.}, as our sole aim is to use the following result.

\begin{theorem}[\cite{Bo15}] \label{Thm:Bo15}
 The smooth complex variety $M_{\overline{Q}}^{\mathfrak{c}}:= \cap_{c\in C} M_{\overline{Q}}^{\mathfrak{c},c}$ is endowed with a structure of quasi-Hamiltonian variety for the moment map 
 \begin{equation} \label{Eq:momap-Bo15}
\Phi^{\mathfrak{c}}:M_{\overline{Q}}^{\mathfrak{c}}\to \Gl_{\nfat},\quad
\Phi^{\mathfrak{c}}(\Xtt)= (\Phi^{\mathfrak{c}}_s(\Xtt))_{s\in I} \quad \text{ for} \quad 
\Phi^{\mathfrak{c}}_s(\Xtt):=\prod_{\substack{c\in C \\ \text{with}\,s\in I_c}} \Phi_{c,s} \in \Gl_{n_s}. 
\end{equation}
\end{theorem}
The moment map in \eqref{Eq:momap-Bo15} depends on an ordering of the colors at each vertex $s\in I$, which is irrelevant up to isomorphism because the structure is obtained by fusion. 
We shall not need to write down explicitly the $2$-form $\omega_B$ giving the quasi-Hamiltonian structure in Theorem \ref{Thm:Bo15}, but it can be obtained by combining \cite[Prop.~5.3]{Bo15} with the quasi-Hamiltonian structure of the higher fission spaces given in \cite{Bo14}.  
By the correspondence of Proposition \ref{Pr:Corr}, there is an associated quasi-Poisson structure for a non-degenerate quasi-Poisson bivector\footnote{It should be possible to write $P_B$ using the results of Li-Bland and \v{S}evera \cite{LBS1,LBS2}, although we are unaware of a precise formula for $P_B$ in the above parametrization of $M_{\overline{Q}}^{\mathfrak{c}}$.} 
$P_B\in \Gamma(M_{\overline{Q}}^{\mathfrak{c}},\bigwedge^2 TM_{\overline{Q}}^{\mathfrak{c}})$.  
For the next result, we extend the color function as $\mathfrak{c}:\overline{Q}\to C$ by putting $\mathfrak{c}(a^\ast)=\mathfrak{c}(a)$ for any $a\in Q$. 

\begin{theorem}  \label{Thm:Bo15-Pencil} 
Fix an ordering on pairs $\mathcal{C}:=\{(c,s) \mid c\in C, \, s\in I_c\}$. 
 Fix $\zbar=(z_{c,r;c',s})_{(c,r)<(c',s)}$ where $z_{c,r;c',s}\in \CC$ for each $(c,r),(c',s)\in \mathcal{C}$ 
 with $(c,r)<(c',s)$. 
Define the following bivector  and $2$-form on $M_{\overline{Q}}^{\mathfrak{c}}$: 
\begin{align}
 \psi_{\zbar}&= \sum_{(c,r)<(c',s)} z_{c,r;c',s}\,  \Inf^{(c,r)} \wedge \Inf^{(c',s)}\,, \label{Eq:psiBo15} \\
 \varpi_{\zbar}&= \sum_{(c,r)<(c',s)} z_{c,r;c',s}\,  \tr \left[\Phi_{c,r}^{-1} \, \dd\Phi_{c,r} \right] \wedge 
 \tr \left[ \Phi_{c',s}^{-1} \, \dd\Phi_{c',s} \right]\,. \label{Eq:varpiBo15}
\end{align}
where for $(c,r)\in \mathcal{C}$ we have 
\begin{align} \label{Eq:Inf-cr}
 \Inf^{(c,r)}:=\sum_{\substack{a\in \overline{Q}\, \text{with}\\ \mathfrak{c}(a)=c,\, t(a)=r}} \tr \left[\Xtt_{a^\ast}\,\partial_{a^\ast}  - \Xtt_{a}\, \partial_a \right] \,. 
\end{align}
Then: 
\begin{enumerate}
 \item the triple $(M_{\overline{Q}}^{\mathfrak{c}}, P_B+\psi_{\zbar},\Phi^{\mathfrak{c}})$ is a non-degenerate Hamiltonian quasi-Poisson variety;
 \item the triple $(M_{\overline{Q}}^{\mathfrak{c}}, \omega_B+\varpi_{\zbar},\Phi^{\mathfrak{c}})$ is a quasi-Hamiltonian variety;
 \item the two triples correspond to one another via \eqref{Eq:corrPOm}. 
\end{enumerate}
\end{theorem} 
\begin{proof}
 This is similar to the proof of Theorem \ref{Thm:Pencil-MQV}: we use Proposition \ref{Pr:Pencil-Cact} on the variety $M_{\overline{Q}}^{\mathfrak{c}}$ considered before fusion of the components for different colors, then conclude with Lemma \ref{Lem:AltPenc}.  
 Before fusion, we see $M_{\overline{Q}}^{\mathfrak{c}}$  as a $G$-variety for $G=\prod_{s\in I} \Gl_{n_s}^{C_{(s)}}$, where $C_{(s)} = \{c\in C \mid s\in I_c\}$.
For fixed $s$ and $c\in C_{(s)}$, the action of the component $\Gl_{n_s}$ of $\Gl_{n_s}^{C_{(s)}}$ corresponding to $c$ is given on a morphism $\Xtt_a:M_{\overline{Q}}^{\mathfrak{c}} \to \End(\CC^{\nfat})$, $a\in \overline{Q}$, by
\begin{equation} \label{Eq:Act-BoPf}
g\cdot \Xtt_a=\left\{
\begin{array}{ll}
\Xtt_a & \mathfrak{c}(a)\neq c  \\
g^{-\delta_{t(a),s}}\Xtt_a g^{\delta_{h(a),s}} & \mathfrak{c}(a)= c
\end{array}
\right. \,, \qquad g \in \Gl_{n_s}\,,
\end{equation}
where we denote by $g$ its embedding under $\Gl_{n_s}\hookrightarrow \Gl_{\nfat}=\prod_{s\in I} \Gl_{n_s}$ (without repeated factors for $c\in C_{(s)}$).
The component of the moment map corresponding to this $c$-th component $\Gl_{n_s}$ of $\Gl_{n_s}^{C_{(s)}}$ is $\Phi_{c,s}$ due to the construction that yields Theorem \ref{Thm:Bo15}, cf. \cite{Bo15}.
Moreover, differentiating \eqref{Eq:Act-BoPf} gives the infinitesimal action of $1\in \CC$ corresponding to the embedding of $\CC^\times$ in the $c$-th component $\Gl_{n_s}$ of $\Gl_{n_s}^{C_{(s)}}$, which coincides with $\Inf^{(c,r)}$ \eqref{Eq:Inf-cr}.
\end{proof}

In analogy with the original theory of Crawley-Boevey and Shaw \cite{CBShaw} given in \ref{ss:MQV}, Boalch defines the \emph{generalized multiplicative quiver variety} $\Mcal_{\overline{Q},\qfat,\theta}^{\mathfrak{c}}$ at a parameter $\qfat\in (\CC^\times)^I$ with stability $\theta\in \Q^I$ by reduction of $M_{\overline{Q}}^{\mathfrak{c}}$ as in Definition \ref{Def:MQV}.
One can then spell out the obvious analogue of Corollary \ref{Cor:Pencil-redMQV} as a consequence of Theorem \ref{Thm:Bo15-Pencil}.

Finally, if $C=Q$ and $\mathfrak{c}=\id_Q$, we get back to the original construction presented in \ref{ss:MQV} as discussed in \cite[\S5]{Bo15}.
Thus, Theorem \ref{Thm:Bo15-Pencil} gives back Theorem \ref{Thm:Pencil-MQV} in that situation.
Remark that the pencil obtained from the generalized theory depends on $\mathcal{C}(\mathcal{C}-1)/2=|Q|(2|Q|-1)$ parameters instead of $r_Q=|Q|(|Q|-1)/2$;
this is because infinitesimal vector fields appear twice with opposite signs, cf. the discussion after Theorem \ref{Thm:Pencil-MQV}.

\subsubsection{Quiver varieties} \label{ss:QuiVar}

For the interested reader, let us note the ``additive'' version of Proposition \ref{Pr:Pencil-Cact}, and how it applies to quiver varieties. 
Recall that a triple $(M,P,\mu)$ is a Hamiltonian Poisson $G$-variety if $P$ is a Poisson bivector and $\mu:M\to \g$ is a $G$-equivariant morphism satisfying $\xi_M=P^\sharp(\dd(\mu,\xi)_\g)$ for all $\xi\in \g$. 
If $P^\sharp_x$ defines an isomorphism $T^\ast_x M\to T_x M$ at each $x\in M$, there corresponds a unique symplectic form $\omega$ such that $P^\sharp\circ \omega^\flat=\Id_{TM}$; in particular $\omega^\flat(\xi_M)=\dd(\mu,\xi)_\g$.  

\begin{proposition}   \label{Pr:Pencil-Ham}
Assume that $(M,P,\mu)$ is a Hamiltonian Poisson $G$-variety, where $G$ admits a decomposition 
$G=G'\times \times_{j=1}^\ell G_j$ such that for each $1\leq j \leq \ell$, 
$\CC^\times$ embeds in the center of $G_j$ as diagonal matrices. 
For each $1\leq j \leq \ell$, denote by $\Inf^{(j)}$ the infinitesimal action of $1\in \CC$ associated with the embedding of $\CC^\times$ into $G_j$.  
Write the moment map as $\mu=(\mu',\mu_1,\ldots,\mu_\ell)$ with 
$\mu':M\to \g'$ and $\mu_j:M\to \g_j$, $1\leq j\leq \ell$. 
Furthermore, fix $\zbar=(z_{ij})_{1\leq i<j\leq \ell} \in \CC^{\ell(\ell-1)/2}$. 
\begin{enumerate} 
 \item The triple $(M,P+\psi_{\zbar},\mu)$ is a Hamiltonian Poisson $G$-variety for $\psi_{\zbar}$ defined by \eqref{Eq:PrPenc-psi}.  
 \item Assume that $P$ is non-degenerate and $\omega$ is the corresponding symplectic form. Let 
\begin{equation} \label{Eq:PrPenc-Omeg}
 \omega_{\zbar}:= \sum_{1\leq i<j\leq \ell} z_{ij}\, \tr(\dd\mu_i) \wedge \tr(\dd\mu_j)\,.
\end{equation}
 Then $P+\psi_{\zbar}$ is non-degenerate, and $(M,\omega+\omega_{\zbar},\mu)$ defines the corresponding (Hamiltonian) symplectic variety. 
\end{enumerate}
\end{proposition}
\begin{proof}
This follows by adapting the proof of Proposition \ref{Pr:Pencil-Cact} to the present setting.  
\end{proof}

Let $Q$ be a quiver and $\nfat$ a dimension vector. 
Rephrasing Nakajima \cite{Nak} with the notation of \ref{ss:MQV}, the $\Gl_{\nfat}$-variety $M_{\overline{Q}}$ admits the following Poisson bivector and symplectic form 
\begin{align}
 P_{\mathtt{qv}}= \sum_{a\in Q} \, \tr[\del_a \wedge \del_{a^\ast}] \,, \quad
 \omega_{\mathtt{qv}}= - \sum_{a\in Q} \, \tr[\dd \Xtt_a \wedge \dd \Xtt_{a^\ast}] \,,
\end{align}
which correspond to one another. Their moment map is given by 
\begin{equation}
 \mu:M_{\overline{Q}}\to \gl_{\nfat}, \quad 
 \mu(\Xtt):= \sum_{a\in Q} (\Xtt_a \Xtt_{a^\ast} - \Xtt_{a^\ast} \Xtt_a)\,.
\end{equation}
We get the following result as a direct application of Proposition \ref{Pr:Pencil-Ham}, even though it is easy to check and certainly known to experts. 

\begin{theorem}   \label{Thm:Pencil-QV}
 Fix $\zbar=(z_{a,b})_{a<b}$ where $z_{a,b}\in \CC$ for each $a,b\in Q$ with $a<b$. 
Define  on $M_{\overline{Q}}$ the bivector $\psi_{\zbar}$ by \eqref{Eq:psiMQV},  and the $2$-form $\omega_{\zbar}$ by   
\begin{equation}
\omega_{\zbar}= \sum_{a<b} z_{a,b}\,  \tr \left[\dd (\Xtt_a \Xtt_{a^\ast})\right] \wedge 
\tr \left[\dd (\Xtt_b \Xtt_{b^\ast})\right] \,. \label{Eq:varomQV}
\end{equation}
Then the triple $(M_{\overline{Q}}, P_{\mathtt{qv}}+\psi_{\zbar},\mu)$ is a non-degenerate Hamiltonian Poisson variety such that  
$(M_{\overline{Q}}, \omega_{\mathtt{qv}}+\omega_{\zbar},\mu)$ is the corresponding Hamiltonian symplectic variety.  
\end{theorem} 

By performing Hamiltonian reduction, we obtain a pencil of compatible Poisson brackets on quiver varieties. 

\begin{remark} \label{Rem:MixedPoi}
It was suggested by A. Alekseev that the Poisson structures defined in the previous two results can be obtained as special instances of a construction due to Lu and Mouquin \cite{LuM}.
Namely, for Proposition \ref{Pr:Pencil-Ham}, see $G$ as a Poisson-Lie group for the zero Poisson structure, so that $\g=\g' \oplus \bigoplus_{j=1}^\ell \g_j$ is a Lie bialgebra with cobracket $\delta_\g$ being trivially the sum of the zero cobrackets on $\g', \g_1,\ldots,\g_\ell$.
If $1^{(j)}\in \g$ denotes the image of $1\in \CC$ induced at the level of Lie algebras by the inclusions $\CC^\times \hookrightarrow G_j \hookrightarrow G$, then $\mathtt{t}=\sum_{i<j} z_{ij} 1^{(i)} \wedge 1^{(j)} \in \wedge^2 \g$ is a mixed twisting element \cite[Def.~3.1]{LuM}. It follows from \cite[Lem.~2.6]{LuM} that $(M,P+\mathtt{t}_M)$ is a Poisson $(\g,\delta_{\g,\mathtt{t}})$-variety for the twisting cobracket $\delta_{\g,\mathtt{t}}=\delta_\g + [\mathtt{t},-]$; but the latter vanishes while the action of $\g$ on $M$ sends $\mathtt{t}$ to $\mathtt{t}_M=\psi_{\zbar}$ \eqref{Eq:PrPenc-psi}, hence $(M,P+\psi_{\zbar})$ is a Poisson $G$-variety.

As for Theorem \ref{Thm:Pencil-QV}, consider for each $a\in Q$ the subquiver
$\Gamma_a= t(a) \underset{a^\ast}{\stackrel{a}{\rightleftarrows}}  h(a)$ of $\overline{Q}$.
We have an action of $G_a:=\Gl_{n_{t(a)}}\times \Gl_{n_{h(a)}}$ on
$Y_a:=M_{\Gamma_a}(n_{t(a)},n_{h(a)})$, cf. \eqref{Eq:Act-Gln}, thus $M_{\overline{Q}} = \times_{a\in Q} Y_a$ inherits an action of $\times_{a\in Q} G_a$ (it is ``before fusion'', see Remark \ref{Rem:MQV-fusion}).
Denote the Lie algebra of $G_a$ by $\g_a$, and see $\g=\oplus_{a\in Q} \g_a$ as a Lie bialgebra for the zero cobracket as above. Let $1^{(a)}$ be the image of $(\Id_{n_{t(a)}},0_{n_{h(a)}})$ under the inclusion $\g_a \hookrightarrow \g$.
Again, $\mathtt{t}=\sum_{a<b} z_{ab} 1^{(a)} \wedge 1^{(b)}$ is mixed twisting and induces $\psi_{\zbar}$ \eqref{Eq:psiMQV}.
Hence $P_{\mathtt{qv}}+\psi_{\zbar}$ is a mixed product Poisson structure on $M_{\overline{Q}}$ for the action of $\times_{a\in Q} G_a$, cf. \cite[Prop.~3.4]{LuM}.
Factoring the action of $\Gl_{\nfat}$ through $\times_{a\in Q} G_a$ by taking diagonal embeddings yields the Poisson $\Gl_{\nfat}$-variety structure considered in Theorem \ref{Thm:Pencil-QV}.
\end{remark}


\section{Noncommutative origin of the pencil} \label{Sec:NC}

It was realised by Van den Bergh \cite{VdB1} that the Poisson bracket obtained on any affine multiplicative quiver variety 
$\Mcal_{\overline{Q},\qfat,0}$ through Corollary \ref{Cor:qPqHam-redMQV} has a noncommutative origin which is defined at the level of the quiver $\overline{Q}$.  
We shall explain this statement in the presence of a deformation parameter as in \ref{ss:DefoMomap}, then we describe how to adapt this point of view for the whole Poisson pencil. 

\medskip 

Fix an associative unital $\CC$-algebra $A$, and consider $A\otimes A$ as a $\CC$-algebra for the product satisfying $(f_1\otimes f_2)(f_3\otimes f_4)=f_1 f_3 \otimes f_2 f_4$, $f_1,\ldots,f_4\in A$. Denote by $\tau_{(12)}\in \End(A\otimes A)$ the operation that exchanges the two tensor factors. 
An \emph{outer double bracket}\footnote{This is simply called a \emph{double bracket} by Van den Bergh \cite{VdB1}. 
We dub this operation as \emph{outer} because the derivation rule of the second argument is defined thanks to the outer $A$-bimodule structure on $A\otimes A$ which is determined by $f_1\ast (g\otimes h) \ast f_2 = f_1 g \otimes h f_2$ for any $f_1,f_2,g,h\in A$. 
Below, we shall use \emph{right double brackets} \cite{FMC} based on the right $A$-bimodule structure $f_1\cdot_r (g\otimes h) \cdot_r f_2 = g \otimes f_1 h f_2$; hence qualifying the type of double bracket removes any possible confusion.}  
is a bilinear map $\dgal{-,-}_o:A\times A \to A\otimes A$ such that for $f,g,h\in A$, 
\begin{subequations}
  \begin{align}
    \dgal{f,g}_o&=-\,\tau_{(12)} \dgal{g,f}_o \,, \\
    \dgal{f,gh}_o&=(g\otimes 1) \dgal{f,h}_o + \dgal{f,g}_o (1\otimes h)\,, \label{Eq:dbr-der1}\\
    \dgal{fg,h}_o&=(1\otimes f) \dgal{g,h}_o + \dgal{f,h}_o (g\otimes 1)\,. \label{Eq:dbr-der2}
  \end{align}
\end{subequations}
If $A$ is a $B$-algebra, then we require the outer double bracket to be $B$-bilinear, i.e. $\dgal{A,B}_o=0$. Due to the two derivation rules \eqref{Eq:dbr-der1}--\eqref{Eq:dbr-der2}, it suffices to specify an outer double bracket on generators of $A$. 
For our purpose, take $A=\CC \overline{Q}$ the path algebra of a double quiver $\overline{Q}$ endowed with a total ordering as in \ref{ss:MQV}; the ordering induces a skewsymmetric map $o(-,-):\overline{Q}\times \overline{Q}\to \{-1,0,1\}$ by setting $o(a,b)=+1$ if $a<b$. 
We see $A$ as a $B$-algebra for $B=\oplus_{s\in I} \CC e_s$. We fix constants $(\gamma_a)_{a\in \overline{Q}}$ as in \ref{ss:DefoMomap} (thus $\gamma_{a^\ast}=\gamma_a$ for all $a\in \overline{Q}$). 
Following \cite[\S3.2]{Fa21} (based on \cite[\S6.7]{VdB1}), we consider the $B$-linear outer double bracket that satisfies for $a,b\in \overline{Q}$: 
\begin{equation}
 \begin{aligned} \label{Eq:dgal-VdB}
  \dgal{a,b}_o&= \epsilon(a) \, \delta_{(b=a^\ast)} \left[ 
  \gamma_a e_{h(a)}\otimes e_{t(a)} + \frac12 ba \otimes e_{t(a)} + \frac12 e_{h(a)} \otimes ab \right] \\
&\quad -\frac12 \, o(a,b) \, e_{t(b)}a \otimes e_{t(a)}b - \frac12\,  o(a^\ast,b^\ast) \, b e_{h(a)} \otimes a e_{h(b)} \\
&\quad + \frac12\, o(a,b^\ast)\, ba \otimes e_{t(a)} + \frac12 \, o(a^\ast,b)\, e_{h(a)} \otimes ab\,.
 \end{aligned}
\end{equation}
By localisation \cite[Prop.~2.5.3]{VdB1}, the outer double bracket descends to $A_Q:= \CC \overline{Q}_{S^\gamma}$ obtained by adding inverses to the elements in $S:=\{\gamma_a 1+aa^\ast \mid a \in \overline{Q}\}$. 

An important result is that the pair $(A_Q, \dgal{-,-}_o)$ is a double quasi-Poisson algebra \cite[Theorem~3.3]{Fa21}. Without unravelling the definition (cf. \cite[\S5.1]{VdB1}) let us only give consequences of this fact.
Put $H_0(A_Q)=A_Q/[A_Q,A_Q]$, which is a vector space coming with a linear map $\tr:A_Q\to H_0(A_Q)$ that sends commutators to zero. 
We get from \cite[Lem.~5.1.3]{VdB1} that $H_0(A_Q)$ is equipped with a Lie bracket, denoted $\brVdB{-,-}$, uniquely determined by 
\begin{equation} \label{Eq:dbr-VdB}
 \brVdB{\tr(f),\tr(g)}= \tr ( \mathrm{m}(\dgal{f,g}_o)), \quad f,g\in A_Q,
\end{equation}
where $\mathrm{m}:A_Q\otimes A_Q \to A_Q$ is the multiplication. In particular, \eqref{Eq:dbr-VdB} extends to a Poisson bracket on $\Sym(H_0(A_Q))$, which we also denote by $\brVdB{-,-}$. 
The following result follows from \cite[\S7.7,7.13]{VdB1} and the fact that the representation space $\Rep(A_Q,\nfat)$ coincides with $M_{\overline{Q}}^\gamma(\nfat)$ for any $\nfat\in \Z^I_{\geq 0}$. 

\begin{proposition} \label{Pr:MorRepAQ-1}
 The morphism of algebras 
\begin{equation} \label{Eq:MorRepAQ}
 \Xtt_{\tr}: \Sym(H_0(A_Q)) \to \CC[M_{\overline{Q}}^\gamma/\!/\Gl_{\nfat}], \quad 
 \tr(a_1 \cdots a_k) \mapsto \tr(\Xtt_{a_1} \cdots \Xtt_{a_k}) \,\, \text{ for }a_1,\ldots,a_k\in \overline{Q}\,,
\end{equation}
is a morphism of Poisson algebras if $\Sym(H_0(A_Q))$ is equipped with $\brVdB{-,-}$ \eqref{Eq:dbr-VdB} and $\CC[M_{\overline{Q}}^\gamma/\!/\Gl_{\nfat}]$ is equipped with $P^\gamma$ \eqref{Eq:qP-MQV-bis}. 
\end{proposition}

\begin{remark} \label{Rem:RepAQ}
 Note that \eqref{Eq:MorRepAQ} is surjective by Le Bruyn-Procesi's theorem \cite{LP90}. Thus, the pair $(\Sym(H_0(A_Q)),\brVdB{-,-})$ is a ``universal'' Poisson algebra, in the sense that it induces the Poisson structure on each variety $M_{\overline{Q}}^\gamma/\!/\Gl_{\nfat}$ (hence on any affine multiplicative quiver variety).
 By \eqref{Eq:dbr-VdB}, all computations can be performed using the pair $(A_Q,\dgal{-,-}_o)$, which is why it can be seen as a \emph{noncommutative origin} for these Poisson structures. 
\end{remark}

To understand the bivector $\psi_{\zbar}$ \eqref{Eq:varpiMQV} at the level of $\overline{Q}$, we need an analogue of Van den Bergh's outer double brackets. 
Taking an arbitrary $A$ as above, A \emph{right double bracket} \cite{FMC} 
is a bilinear map $\dgal{-,-}_r:A\times A \to A\otimes A$ such that for $f,g,h\in A$, 
\begin{subequations}
  \begin{align}
    \dgal{f,g}_r&=-\,\tau_{(12)} \dgal{g,f}_r \,, \label{Eq:dbrR-skew} \\
    \dgal{f,gh}_r&=(1\otimes g) \dgal{f,h}_r + \dgal{f,g}_r (1\otimes h)\,, \label{Eq:dbrR-der1}\\
    \dgal{fg,h}_r&=(f\otimes 1) \dgal{g,h}_r + \dgal{f,h}_r (g\otimes 1)\,. \label{Eq:dbrR-der2}
  \end{align}
\end{subequations}
As in the outer case, we can assume $B$-bilinearity if $A$ is a $B$-algebra, and it suffices to specify the right double bracket on generators.  
Examples of right double brackets can be easily constructed using the following result. 

\begin{lemma} \label{Lem:DbrR}
 Let $Q'$ be an arbitrary quiver with a total ordering $<'$ on its arrows. 
For any $a,b\in Q'$ with $a<' b$, fix $y_{ab} \in \CC$ and extend these constants to a skewsymmetric matrix $(y_{ab})_{a,b\in Q'}$ by setting $y_{aa}=0$ and $y_{ab}=-y_{ba}$ if $b<'a$. 

\begin{enumerate}
 \item There is a unique $B'$-linear right double bracket on $\CC Q'$, where $B'=\oplus_{s\in I'} \CC e_s$, such that for $a,b\in Q'$, $\dgal{a,b}_r= y_{ab} \, a \otimes b$. 
 \item Given a right double bracket $\dgal{-,-}_r$ on an algebra $A$, introduce the multilinear map
 \begin{equation}
  \begin{aligned}  \label{Eq:Trilin}
\dgal{-,-,-}_r&:    A^{\times 3} \to A^{\otimes 3}\,, \\
\dgal{-,-,-}_r&:= \sum_{k=0,1,2} \tau_{(123)}^{k} \circ (\dgal{-,-}_r \otimes \id_A) \circ (\id_A \otimes \dgal{-,-}_r) \circ \tau_{(123)}^{-k}\,,
  \end{aligned}
 \end{equation}
where, for a permutation $\rho \in S_3$, we denote by $\tau_\rho: A^{\otimes 3} \to A^{\otimes 3}$ the corresponding action by permutation of tensor factors. 
Then the right double bracket from part $\mathrm{(1)}$ on $\CC Q'$  is $(12)$-weak Poisson \cite[Def.~2.7]{FMC}, i.e.
\begin{equation} \label{Eq:TrilinWk}
 \dgal{-,-,-}_r - \tau_{(12)}^{-1} \circ \dgal{-,-,-}_r \circ \tau_{(12)} = 0\,.
\end{equation}
\item The right double bracket from $\mathrm{(1)}$  descends to any localisation $A=\CC Q'_S$ at a set $S\subset \CC Q'$,  where it is $(12)$-weak Poisson. 
\end{enumerate}
\end{lemma}
\begin{proof}
 For (1), it suffices to check that the skewsymmetry rule \eqref{Eq:dbrR-skew} holds on generators. But this is a direct consequence of the defining identity  $\dgal{a,b}_r= y_{ab} \, a \otimes b$, $a,b\in Q'$, because $(y_{ab})_{a,b\in Q'}$ is skewsymmetric. (Note that, by requiring $B'$-linearity, $\dgal{\CC Q',e_s}_r=0=\dgal{e_s,\CC Q'}_r$ for $s\in I'$.)
 
For (2), we use that the map $\dgal{-,-,-}_r - \tau_{(12)}^{-1} \circ \dgal{-,-,-}_r \circ \tau_{(12)}$ is a derivation in each argument (for specific $A$-bimodule structures on $A^{\otimes 3}$) by \cite[Lem.~3.12]{FMC}. 
Hence we only have to verify \eqref{Eq:TrilinWk} on generators $a,b,c\in Q'$, cf. \cite[Lem.~2.11]{FMC}. 
We first compute 
\begin{equation}
(\dgal{-,-}_r \otimes \id_A) \circ (\id_A \otimes \dgal{-,-}_r) (a,b,c)
=y_{bc} \dgal{a,b}_r \otimes c = y_{ab} y_{bc} \, a\otimes b\otimes c\,.
\end{equation}
Therefore, by \eqref{Eq:Trilin}, we obtain
\begin{equation}
 \dgal{a,b,c}_r= \big( y_{ab} y_{bc} + y_{bc} y_{ca} + y_{ca} y_{ab} \big)\, a\otimes b\otimes c\,.
\end{equation}
This gives at once 
\begin{equation}
\tau_{(12)} \dgal{b,a,c}_r= \big( y_{ba} y_{ac} + y_{ac} y_{cb} + y_{cb} y_{ba} \big)\, a\otimes b\otimes c \,,
\end{equation}
which equals $\dgal{a,b,c}_r$ as $(y_{ab})_{a,b\in Q'}$ is skewsymmetric. Thus \eqref{Eq:TrilinWk} holds on the triple $(a,b,c)$. 

For (3), the right double bracket uniquely descends to $\CC Q'_S$ since we need  
$$\dgal{a,s^{-1}}_r=-(1\otimes s^{-1}) \dgal{a,s}_r(1\otimes s^{-1}), \quad 
\dgal{s^{-1},a}_r=-( s^{-1} \otimes 1) \dgal{s,a}_r(s^{-1} \otimes 1),$$ 
as a consequence of \eqref{Eq:dbrR-der1}--\eqref{Eq:dbrR-der2} for any $a\in \CC Q'$ and $s\in S$. 
Using again that the map on the left-hand side of \eqref{Eq:TrilinWk} is a derivation by \cite[Lem.~3.12]{FMC}, its value on $(\CC Q'_S)^{\times 3}$ can be computed using its value on $(\CC Q')^{\times 3}$, where it vanishes.
\end{proof}

For our purpose, we return to the case of $\overline{Q}$ with a total ordering as in \ref{ss:MQV}. For any $a,b\in Q$ with $a<b$, fix $z_{ab}\in \CC$, and extend these constants to a skewsymmetric matrix $(z_{ab})_{a,b\in \overline{Q}}$, which is symmetric under the involution $a\mapsto a^\ast$, by setting for any $a,b\in Q$: 
\begin{equation}
 \begin{aligned} \label{Eq:zab-Ext}
   &z_{aa}=0, \quad z_{ab}=-z_{ba} \quad \text{ if }b<a\,, \\
 &z_{a^\ast b}=z_{a b^\ast}=z_{a^\ast b^\ast}=z_{ab}\,.
 \end{aligned}
\end{equation}

We can then use Lemma \ref{Lem:DbrR} with $Q'=\overline{Q}$ to get a right double $(12)$-weak Poisson bracket on $A_Q=\CC \overline{Q}_{S^\gamma}$  
which is $B$-linear and uniquely determined by 
\begin{equation} \label{Eq:dbrR-Q}
 \dgal{a,b}_r= \epsilon(a)\, \epsilon(b)\, z_{ab} \, a \otimes b\,, \quad a,b\in \overline{Q}.
\end{equation}
Indeed, this corresponds to taking the constants $y_{ab}= \epsilon(a) \epsilon(b) z_{ab}$. 
Thanks to \cite[Prop.~3.15]{FMC}, the $(12)$-weak Poisson property guarantees that  $\Sym(H_0(A_Q))$ is equipped with a Poisson bracket, denoted $\br{-,-}_{\zbar}$, uniquely determined by 
\begin{equation} \label{Eq:dbr-Z}
 \br{\tr(f),\tr(g)}_{\zbar}= (\tr\otimes \tr) (\dgal{f,g}_r), \quad f,g\in A_Q\,.
\end{equation}
Remark that, as opposed to \eqref{Eq:dbr-VdB}, the bracket \eqref{Eq:dbr-Z} maps $H_0(A_Q)\times H_0(A_Q)$ into $\Sym^2(H_0(A_Q))$, hence it is \emph{not} obtained from a Lie bracket on $H_0(A_Q)$.  
The following result is the analogue of Proposition \ref{Pr:MorRepAQ-1}. 

\begin{proposition} \label{Pr:MorRepAQ-2}
 The morphism of algebras $\Xtt_{\tr}: \Sym(H_0(A_Q)) \to \CC[M_{\overline{Q}}^\gamma/\!/\Gl_{\nfat}]$ \eqref{Eq:MorRepAQ} 
is a morphism of Poisson algebras if $\Sym(H_0(A_Q))$ is equipped with $\br{-,-}_{\zbar}$ \eqref{Eq:dbr-Z} and $\CC[M_{\overline{Q}}^\gamma/\!/\Gl_{\nfat}]$ is equipped with $\psi_{\zbar}$  \eqref{Eq:psiMQV}. 
\end{proposition}
\begin{proof}
We can equip $\CC[M_{\overline{Q}}^\gamma]$ with a unique antisymmetric biderivation satisfying for $a,b\in \overline{Q}$, 
\begin{equation} \label{Eq:PsiExpl}
 \br{(\Xtt_a)_{ij},(\Xtt_b)_{kl}}:=\epsilon(a)\, \epsilon(b)\, z_{ab} \,  (\Xtt_a)_{ij}\, (\Xtt_b)_{kl}\,.
\end{equation}
By \cite[Theorem~4.11]{FMC} and \eqref{Eq:dbrR-Q}, the operation $\br{-,-}$ is a Poisson bracket, 
and it descends to $\CC[M_{\overline{Q}}^\gamma/\!/\Gl_{\nfat}]$. 
We remark in view of \cite[(4.7)]{FMC} that 
 $\Xtt_{\tr}$ is a morphism of Poisson algebras if $\CC[M_{\overline{Q}}^\gamma/\!/\Gl_{\nfat}]$ is equipped with that Poisson bracket and $\Sym(H_0(A_Q))$ is equipped with $\br{-,-}_{\zbar}$~\eqref{Eq:dbr-Z}. 
Thus, we need to verify that \eqref{Eq:PsiExpl} coincides with the expression for $\psi_{\zbar}$  \eqref{Eq:psiMQV} evaluated on generators $(\Xtt_a)_{ij}$ ($a\in \overline{Q}$) of $\CC[M_{\overline{Q}}^\gamma]$. Using  \eqref{Eq:psiMQV}, we compute  
\begin{equation}
\br{(\Xtt_a)_{ij},(\Xtt_b)_{kl}}= \left\{ 
\begin{array}{ll}
z_{ab}\,(\Xtt_a)_{ij} (\Xtt_b)_{kl}& \text{if }a,b\in Q \text{ with }a<b, \\
z_{a^\ast b^\ast}\,(\Xtt_a)_{ij} (\Xtt_b)_{kl}& \text{if }a^\ast,b^\ast\in Q \text{ with }a^\ast<b^\ast, \\
-z_{a^\ast b}\,(\Xtt_a)_{ij} (\Xtt_b)_{kl}& \text{if }a^\ast,b\in Q \text{ with }a^\ast<b, \\
- z_{ab^\ast}\,(\Xtt_a)_{ij} (\Xtt_b)_{kl}& \text{if }a,b^\ast\in Q \text{ with }a<b^\ast.
\end{array}
\right.
\end{equation}
By definition of the skewsymmetric matrix $(z_{ab})_{a,b\in \overline{Q}}$ using \eqref{Eq:zab-Ext}, this is equivalent to \eqref{Eq:PsiExpl}. 
\end{proof}

\begin{theorem} \label{Thm:MorRepAQ}
Fix $z_0\in \CC$. 
Then the morphism of algebras $\Xtt_{\tr}: \Sym(H_0(A_Q)) \to \CC[M_{\overline{Q}}^\gamma/\!/\Gl_{\nfat}]$ \eqref{Eq:MorRepAQ} 
is a morphism of Poisson algebras if $\Sym(H_0(A_Q))$ is equipped with $z_0 \brVdB{-,-}+\br{-,-}_{\zbar}$, cf. \eqref{Eq:dbr-VdB} and \eqref{Eq:dbr-Z}, while $\CC[M_{\overline{Q}}^\gamma/\!/\Gl_{\nfat}]$ is equipped with $z_0 P^\gamma+\psi_{\zbar}$, cf. \eqref{Eq:qP-MQV-bis} and  \eqref{Eq:psiMQV}. 
\end{theorem}
\begin{proof}
If we combine Propositions \ref{Pr:MorRepAQ-1} and \ref{Pr:MorRepAQ-2}, the only part of the statement which is unclear is that the linear combination $z_0 \brVdB{-,-}+\br{-,-}_{\zbar}$ is a Poisson bracket. Since we know that it is true if $z_0=0$ or $\zbar=0$, it suffices to show by rescaling that $\brVdB{-,-}+\br{-,-}_{\zbar}$ is a Poisson bracket on $\Sym(H_0(A_Q))$. 
Using again the known cases $z_0=0$ and $\zbar=0$, we are left to check that 
\begin{equation} \label{Eq:ThmRep1}
 \br{f, \brVdB{g,h}}_{\zbar} + \brVdB{f, \br{g,h}_{\zbar}} + \text{cyclic perm. } = 0\,,
\end{equation}
for arbitrary $f,g,h \in \Sym(H_0(A_Q))$. (Here, cyclic perm. means that we add the $4$ terms obtained by replacing $(f,g,h)$ by $(g,h,f)$ and $(h,f,g)$.) We shall check it for elements of the form 
\begin{align*}
 f=\tr(a_1 \cdots a_n), \quad g=\tr(b_1 \cdots b_l), \quad h=\tr(c_1 \cdots c_m), 
\end{align*}
with $a_i,b_{j},c_{k} \in \overline{Q}$ and $n,l,m\geq 1$, from which the result follows. 

Given an element $d\in A\otimes A$, we use Sweedler's notation and write $d=d'\otimes d''$ (this abuse of notation is not harmful since all the operations considered below are linear). 
To write $\br{g,h}_{\zbar}$, we start by computing using the derivation rules \eqref{Eq:dbrR-der1}--\eqref{Eq:dbrR-der2} and \eqref{Eq:dbrR-Q} 
\begin{align*}
\dgal{b_1 \cdots b_l,c_1 \cdots c_m}_r 
=&\sum_{j=1}^l \sum_{k=1}^m  b_1 \cdots b_{j-1} \dgal{b_j,c_k}_r' b_{j+1} \cdots b_{l} \otimes 
c_1 \cdots c_{k-1} \dgal{b_j,c_k}_r'' c_{k+1} \cdots c_{m} \\ 
=&\sum_{j=1}^l \sum_{k=1}^m \epsilon(b_j) \epsilon(c_k) z_{b_j,c_k} \, b_1 \cdots b_l \otimes c_1 \cdots c_m
\end{align*}
Therefore, we get after applying $\tr \otimes \tr$ to this expression that \eqref{Eq:dbr-Z} yields  
\begin{equation} \label{Eq:ThmRep-brVdB}
 \br{g,h}_{\zbar}= \mathtt{c}_{g,h}\, g\, h\,, \quad 
 \mathtt{c}_{g,h}:=\sum_{j=1}^l \sum_{k=1}^m \epsilon(b_j) \epsilon(c_k) z_{b_j,c_k}\,.
\end{equation}
Let us define similarly $\mathtt{c}_{f_1,f_2}$ for any $f_1,f_2 \in \{f,g,h\}$, 
noting that $\mathtt{c}_{f_1,f_2}=-\mathtt{c}_{f_2,f_1}$ by skewsymmetry of $(z_{ab})_{a,b\in \overline{Q}}$. 
We obtain by the derivation rule of the Poisson bracket $\brVdB{-,-}$ that 
\begin{equation*} 
 \brVdB{f, \br{g,h}_{\zbar}}= \mathtt{c}_{g,h}\left( \brVdB{f,g} h + g \brVdB{f,h} \right) \,,
\end{equation*}
from which we can write \eqref{Eq:ThmRep1} in the equivalent form 
\begin{equation} \label{Eq:ThmRep2}
 \br{h, \brVdB{f,g}}_{\zbar} + \mathtt{c}_{g,h}\brVdB{f,g} \, h + \mathtt{c}_{h,f}  \brVdB{g,f} \, h + \text{cyclic perm. } = 0\,.
\end{equation}
In particular, this identity is satisfied provided that we can show 
\begin{equation} \label{Eq:ThmRep3}
 \br{h, \brVdB{f,g}}_{\zbar} =(-\mathtt{c}_{g,h}+ \mathtt{c}_{h,f}) \brVdB{f,g} \, h \,,
\end{equation}
as the 3 displayed terms in \eqref{Eq:ThmRep2} vanish if \eqref{Eq:ThmRep3} holds, and therefore their cyclic permutations also vanish. 

To derive \eqref{Eq:ThmRep3}, we start by computing using the derivation rules \eqref{Eq:dbr-der1}--\eqref{Eq:dbr-der2},   
\begin{align*}
\dgal{a_1 \cdots a_n,b_1 \cdots b_l}_o 
=&\sum_{i=1}^n \sum_{j=1}^l  b_1 \cdots b_{j-1}  \dgal{a_i,b_j}_o' a_{i+1} \cdots a_{n} \otimes 
a_1 \cdots a_{i-1} \dgal{a_i,b_j}_o'' b_{j+1} \cdots b_{l}\,.
\end{align*}
Thus, we deduce from \eqref{Eq:dbr-VdB} that 
\begin{equation}
 \begin{aligned}  \label{Eq:ThmRep-br-fgF}
  \brVdB{f,g}=& \sum_{i=1}^n \sum_{j=1}^l  \, \mathcal{F}_{i,j}, \qquad \text{ where }\\
 \mathcal{F}_{i,j}:=& \tr(\dgal{a_i,b_j}_o' a_{i+1} \cdots a_{n} a_1 \cdots a_{i-1} \dgal{a_i,b_j}_o'' b_{j+1} \cdots b_{l} b_1 \cdots b_{j-1})\,.
 \end{aligned}
\end{equation}
Note that $\mathcal{F}_{i,j}$ is made of 7 terms, cf. \eqref{Eq:dgal-VdB}, which we label $\mathcal{F}_{i,j}^{(1)}, \ldots, \mathcal{F}_{i,j}^{(7)}$, respectively. (These terms could be zero.) Inside the trace for $\mathcal{F}_{i,j}^{(r)}$, $r\neq 1$, all the factors $a_{i'},b_{j'}$ appearing in $f$ and $g$ are present. Therefore, by an argument similar to the one that led to \eqref{Eq:ThmRep-brVdB}, we get  
\begin{equation}
 \br{h, \mathcal{F}_{i,j}^{(r)}}_{\zbar}= \tilde{c}\, h\, \mathcal{F}_{i,j}^{(r)} \,, \quad 
 \tilde{c}:=\sum_{k=1}^m \epsilon(c_k) \left(\sum_{i'=1}^n \epsilon(a_{i'}) z_{c_k,a_{i'}}+ \sum_{j'=1}^l \epsilon(b_{j'}) z_{c_k,b_{j'}} \right)\,,
\end{equation}
for $r=2,\ldots,7$. 
In particular, observe that $\tilde{c}=\mathtt{c}_{h,f}+\mathtt{c}_{h,g}=-\mathtt{c}_{g,h}+ \mathtt{c}_{h,f}$. For $r=1$, we get that the 2 factors $a_i,b_j$ are missing inside the trace for $\mathcal{F}_{i,j}^{(1)}$, hence   
\begin{equation}
 \br{h, \mathcal{F}_{i,j}^{(1)}}_{\zbar}= \hat{c}\, h\, \mathcal{F}_{i,j}^{(1)} \,, \quad 
 \hat{c}:=\sum_{k=1}^m \epsilon(c_k) \left(\sum_{\substack{i'=1\\i'\neq i}}^n \epsilon(a_{i'}) z_{c_k,a_{i'}}
 + \sum_{\substack{j'=1\\j'\neq j}}^l \epsilon(b_{j'}) z_{c_k,b_{j'}} \right)\,.
\end{equation} 
This time, we can write 
\begin{align*}
 \hat{c}&=\mathtt{c}_{h,f}-\sum_{k=1}^m \epsilon(c_k) \epsilon(a_{i}) z_{c_k,a_{i}}
+\mathtt{c}_{h,g} -\sum_{k=1}^m \epsilon(c_k) \epsilon(b_{j}) z_{c_k,b_{j}} \\
&=-\mathtt{c}_{g,h}+ \mathtt{c}_{h,f} - \sum_{k=1}^m \epsilon(c_k) (\epsilon(b_{j}) z_{c_k,b_j}+\epsilon(a_{i}) z_{c_k,a_{i}})\,.
\end{align*}
Note that $\mathcal{F}_{i,j}^{(1)}$ contains a factor $\delta_{(b_j=a_i^\ast)}$, cf. the first term of \eqref{Eq:dgal-VdB}. As we can write 
$$\delta_{(b_j=a_i^\ast)} \sum_{k=1}^m \epsilon(c_k) (\epsilon(b_{j}) z_{c_k,b_j}+\epsilon(a_{i}) z_{c_k,a_{i}}) 
= \delta_{(b_j=a_i^\ast)} \sum_{k=1}^m \epsilon(c_k) \epsilon(a_{i}) (- z_{c_k,a_i^\ast}+ z_{c_k,a_{i}}) \,,$$
this expression vanishes since $(z_{ab})$ is symmetric in each index under the involution $a\mapsto a^\ast$ by \eqref{Eq:zab-Ext}. 
Therefore, $\hat{c}\, \mathcal{F}_{i,j}^{(1)} = (-\mathtt{c}_{g,h}+ \mathtt{c}_{h,f}) \mathcal{F}_{i,j}^{(1)}$, so that 
$\br{h, \mathcal{F}_{i,j}^{(r)}}_{\zbar}= (-\mathtt{c}_{g,h}+ \mathtt{c}_{h,f})\, h\, \mathcal{F}_{i,j}^{(r)}$ for each $1\leq r\leq 7$. 
In particular, $\br{h, \mathcal{F}_{i,j}}_{\zbar}= (-\mathtt{c}_{g,h}+ \mathtt{c}_{h,f})\, h\, \mathcal{F}_{i,j}$. Combining these observations with \eqref{Eq:ThmRep-br-fgF}, we find 
\begin{equation*}
\br{h,\brVdB{f,g}}_{\zbar}  
= (-\mathtt{c}_{g,h}+ \mathtt{c}_{h,f}) \sum_{i=1}^n \sum_{j=1}^l \, h\,  \mathcal{F}_{i,j} 
= (-\mathtt{c}_{g,h}+ \mathtt{c}_{h,f}) \, h\, \brVdB{f,g}\,,
\end{equation*}
which is precisely \eqref{Eq:ThmRep3}. 
\end{proof}

As a consequence of Theorem \ref{Thm:MorRepAQ}, we can repeat Remark \ref{Rem:RepAQ} and see the family of Poisson brackets $(z_0\brVdB{-,-}+\br{-,-}_{\zbar})$ on $\Sym(H_0(A_Q))$ as the universal Poisson pencil that induces the one on the varieties $M_{\overline{Q}}^\gamma/\!/\Gl_{\nfat}$ defined by taking invariants in Theorem \ref{Thm:Pencil-MQV-bis}. 
In fact, any computation can be performed using the operations $(\dgal{-,-}_o,\dgal{-,-}_r)$ on $A_Q$.

\begin{remark}
An analogue of Theorem \ref{Thm:MorRepAQ} in the case of character varieties (cf.~\ref{ss:CharVar}) can be obtained by replacing Van den Bergh's outer double (quasi-Poisson) bracket on $A_Q$ by the one introduced by Massuyeau-Turaev in \cite{MT14} on 
 $\CC\langle a_1^{\pm1},(a_1^\ast)^{\pm1}, \ldots,a_g^{\pm1},(a_g^\ast)^{\pm1},z_1^{\pm1},\ldots,z_r^{\pm1}\rangle$. 
The case of quiver varieties (cf.~\ref{ss:QuiVar}) is obtained similarly by using Van den Bergh's outer double Poisson bracket on $\CC \overline{Q}$ given in \cite[\S6.3-6.4]{VdB1}. 
Alternatively note that the outer double bracket on $\CC \overline{Q}$ in the Poisson case reduces to the first constant term appearing in \eqref{Eq:dgal-VdB}, thus we can formally derive this case from Theorem \ref{Thm:MorRepAQ} by setting $\gamma_a=\tilde{\gamma}\in \CC^\times$ for each $a\in \overline{Q}$ and considering the family $(z_0\tilde{\gamma}^{-1}\brVdB{-,-}+\br{-,-}_{\zbar})$ in the limit $\tilde{\gamma} \to \infty$.
For generalized MQV (cf.~\ref{ss:GenMQV}), it is conjectured in \cite{FFern} that their Poisson structure is induced by an outer double bracket;
as this problem is still open, exhibiting a universal Poisson pencil is currently beyond reach.
\end{remark}


\section{Application to the spin RS phase space}  \label{S:ApplRS}

In \cite{CF}, the authors constructed the phase space of the spin Ruijsenaars-Schneider (RS) system of Krichever and Zabrodin \cite{KrZ} as a multiplicative quiver variety. We revisit this result with the pencil constructed in Theorem \ref{Thm:Pencil-MQV} in order to bridge the gap with the alternative construction of the phase space from \cite{AO} where a \emph{different} Poisson structure is derived. 

\subsection{General construction}  \label{S:sRS-Gen}

We recall the setting of \cite{CF}. 
Fix $d\geq 2$. 
Consider the quiver $Q_d$ consisting of the vertex set $I=\{0,\infty\}$ and the arrows 
$x:0\to 0$, $v_\alpha:\infty \to 0$ for $1\leq \alpha \leq d$. 
On the double $\overline{Q}_d$, we take the ordering $x<x^\ast<v_1<v_1^\ast <\ldots <v_d < v_d^\ast$. 

Let $\gamma:=(\gamma_x,\gamma_\alpha)\in \CC^{|Q_d|}$ be given by $\gamma_x=0$, $\gamma_\alpha=1$ for $1\leq \alpha \leq d$.
For $n\geq1$ and $\nfat=(n_0,n_\infty)=(n,1)$, 
we form the variety $M_{\overline{Q}_d}^{\gamma}$ as in \ref{ss:DefoMomap}. 
Explicitly, if $X,Z,V_\alpha,W_\alpha$ denote, respectively, the nonzero blocks of the matrices $\Xtt_x,\Xtt_{x^\ast},\Xtt_{v_\alpha},\Xtt_{v_\alpha^\ast}$ (where $1\leq \alpha \leq d$), we have 
\begin{equation}
 \begin{aligned} \label{Eq:Param-VW}
M_{\overline{Q}_d}^{\gamma} :=& \{(X,Z,V_\alpha,W_\alpha) \mid \det(X),\det(Z),\det(\Id_n+W_\alpha V_\alpha)\neq 0,\,\, 1\leq \alpha \leq d\} \\
\subset & \{X,Z\in \Gl_n(\CC),\,\, V_1,\ldots,V_d\in \Mat_{1\times n}(\CC),\,\, W_1,\ldots,W_d \in \Mat_{n\times 1}(\CC)\, \} . 
 \end{aligned}
\end{equation}
The action of the group $\Gl_n(\CC)\times \CC^\times$ on the variety is given by 
\begin{equation} \label{Eq:Act-RS}
 (g,\lambda)\cdot (X,Z,V_\alpha,W_\alpha) = (gXg^{-1},gZg^{-1},\lambda V_\alpha g^{-1},g W_\alpha \lambda^{-1})\,.
\end{equation}
Since the group $\{(\lambda \Id_n,\lambda) \mid \lambda \in \CC^\times\}$ acts trivially, we shall only focus on\footnote{As we omit the action of an abelian group, this does not change the description of the quasi-Poisson structure given subsequently.} the action of $\Gl_n(\CC)$.  
By Theorem \ref{Thm:qPqHam-MQV-bis}, $M_{\overline{Q}_d}^{\gamma}$ is a smooth complex variety of dimension $2n^2+2nd$ 
with the non-degenerate quasi-Poisson bivector $P$ \eqref{Eq:qP-MQV-bis} and the corresponding moment map 
\begin{equation} \label{Eq:Momap-RS}
 \Phi:M_{\overline{Q}_d}^{\gamma}\to \Gl_n,\quad
\Phi(X,Z,V_\alpha,W_\alpha):= XZX^{-1}Z^{-1}\, \prod_{1\leq \alpha \leq d} (\Id+W_\alpha V_\alpha)^{-1}\,,
\end{equation}
where we write factors in the product from left to right with increasing indices $1\leq \alpha\leq d$. 
The quasi-Poisson bracket corresponding to $P$ is explicitly spelled out in \cite[Eq.~(2.8)]{CF}. 
(We shall not use the corresponding $2$-form given in \eqref{Eq:omega-MQV-bis}.)

A useful parametrization of $M_{\overline{Q}_d}^{\gamma}$ is given as follows. 
For $1\leq \alpha \leq d$, introduce 
\begin{equation} \label{Eq:Def-AB}
A_\alpha := W_\alpha\,, \quad 
B_\alpha := V_\alpha (\Id_n+W_{\alpha-1} V_{\alpha-1}) \cdots (\Id_n+W_1V_1)Z\,,
\end{equation}
which are easily seen to satisfy 
\begin{equation} \label{Eq:Rel-VWAB}
 (\Id_n+W_{\alpha} V_{\alpha}) \cdots (\Id_n+W_1V_1)Z = Z+ A_1B_1 + \cdots + A_\alpha B_\alpha\,.
\end{equation}
In particular, the right-hand side $\mathcal{Z}_\alpha:=Z+ A_1B_1 + \cdots + A_\alpha B_\alpha$ is invertible. Thus, 
\begin{equation}
 \begin{aligned} \label{Eq:Param-AB}
M_{\overline{Q}_d}^{\gamma} =& \{(X,Z,A_\alpha,B_\alpha) \mid \det(X),\det(Z),\det(\mathcal{Z}_\alpha)\neq 0,\,\, 1\leq \alpha \leq d\} \\
\subset & \{X,Z\in \Gl_n(\CC),\,\, A_1,\ldots,A_d \in \Mat_{n\times 1}(\CC),\,\,B_1,\ldots,B_d\in \Mat_{1\times n}(\CC) \, \} . 
 \end{aligned}
\end{equation}
With that parametrization, the action of $\Gl_n(\CC)$ is obtained from \eqref{Eq:Act-RS} by substituting $A_\alpha,B_\alpha$ for $W_\alpha,V_\alpha$, respectively, while the moment map reads 
\begin{equation} \label{Eq:Momap-RS-2}
 \Phi:M_{\overline{Q}_d}^{\gamma}\to \Gl_n,\quad
\Phi(X,Z,A_\alpha,B_\alpha):= XZX^{-1}\, (Z+ A_1B_1 + \cdots + A_d B_d)^{-1}\,.
\end{equation}
The quasi-Poisson bracket $P$ can be written in terms of those elements by combining Lemma~3.3 and (3.15) from \cite{CF}. 

\medskip

To perform reduction, choose $q\in \CC^\times$ such that $q^k\neq 1$ for $1\leq k\leq n$. 
Then the action of $\Gl_n(\CC)$ that factors through \eqref{Eq:Act-RS} is free\footnote{This requires the condition on $q$. If $q^k=1$, 
then $\Phi^{-1}(q\Id_n)$ contains the points $(X_{(k)},Z_{(k)},0,0)$ for any $k$-dimensional representation $x\mapsto X_{(k)}$, $z\mapsto Z_{(k)}$, of the quantum torus $\CC\langle x^{\pm1},z^{\pm1}\rangle/(xz-qzx)$, which have stabilizers of dimension $\geq1$ as, e.g., $\lambda \Id_n$ acts trivially. 
This condition is not discussed in \cite{AO}, so that their description works only on the smooth locus of the reduced space \eqref{Eq:Cndq} when $q$ is a $k$-th root of unity for some $1\leq k \leq n$.} on the subvariety $\Phi^{-1}(q\Id_n)$, which is denoted $\mathcal{M}_{n,d,q}^\times$ in \cite{CF}.
The affine GIT quotient $\Phi^{-1}(q\Id_n)/\!/\Gl_n(\CC)$ (in fact geometric quotient) corresponding to taking trivial stability parameter $\theta$ in Definition \ref{Def:MQV} is therefore a smooth complex variety of dimension $2nd$, cf. \cite{CF}. 
We set 
\begin{equation} \label{Eq:Cndq}
 \Ccal_{n,d,q}^\times := \Phi^{-1}(q\Id_n)/\!/\Gl_n(\CC)\,, 
\end{equation}
which we call the \emph{spin Ruijsenaars-Schneider (RS) phase space}. 
We get a non-degenerate Poisson bracket on $\Ccal_{n,d,q}^\times$ defined by quasi-Poisson reduction from the bivector $P$.

\subsubsection{The pencil} \label{ss:RS-pencil}

For $\zbar=(z_{\alpha\beta})\in \CC^{d(d-1)/2}$, we form the bivector $\psi_{\zbar}$ on $M_{\overline{Q}_d}^{\gamma}$ using \eqref{Eq:psiMQV} where we set $z_{v_\alpha,v_\beta}=z_{\alpha\beta}$ and $z_{x,v_\alpha}=0$ for each $1\leq \alpha <\beta \leq d$. 
For convenience, we extend $\zbar$ to an antisymmetric $d\times d$ matrix with entries $(z_{\alpha\beta})_{1\leq \alpha,\beta \leq d}$ by setting $z_{\alpha \beta}=-z_{\beta \alpha}$ if $\alpha>\beta$ and $z_{\alpha\alpha}=0$.  
The corresponding bracket, also denoted $\psi_{\zbar}$ by abusing notation, is explicitly given by 
\begin{subequations} \label{Eq:Psi-RS}
 \begin{align}  
\psi_{\zbar}(X_{ij},-)&=0\,, \qquad \psi_{\zbar}(Z_{ij},-)=0\,, \label{Eq:Psi-RSa} \\
\psi_{\zbar}((V_\alpha)_i,(V_\beta)_j)&=z_{\alpha \beta} \, (V_\alpha)_i (V_\beta)_j\,, \quad 
\psi_{\zbar}((W_\alpha)_i,(W_\beta)_j)=z_{\alpha \beta} \, (W_\alpha)_i (W_\beta)_j\,, \label{Eq:Psi-RSb} \\ 
\psi_{\zbar}((V_\alpha)_i,(W_\beta)_j)&=-z_{\alpha \beta} \, (V_\alpha)_i (W_\beta)_j\,, \quad
\psi_{\zbar}((W_\alpha)_i,(V_\beta)_j)=-z_{\alpha \beta} \, (W_\alpha)_i (V_\beta)_j\,, \label{Eq:Psi-RSc}
 \end{align}
\end{subequations}
where $X_{ij},Z_{ij},(V_\alpha)_i,(W_\alpha)_i$ are the ``matrix entry'' functions defined in the obvious way.  
In terms of the alternative parametrization that uses \eqref{Eq:Def-AB}, we still have \eqref{Eq:Psi-RSa}  as well as 
\begin{subequations} 
 \begin{align}   
\psi_{\zbar}((A_\alpha)_i,(A_\beta)_j)&=z_{\alpha \beta} \, (A_\alpha)_i (A_\beta)_j\,, \quad 
\psi_{\zbar}((B_\alpha)_i,(B_\beta)_j)=z_{\alpha \beta} \, (B_\alpha)_i (B_\beta)_j\,, \label{Eq:Psi-RSg} \\ 
\psi_{\zbar}((A_\alpha)_i,(B_\beta)_j)&=-z_{\alpha \beta} \, (A_\alpha)_i (B_\beta)_j\,, \quad
\psi_{\zbar}((B_\alpha)_i,(A_\beta)_j)=-z_{\alpha \beta} \, (B_\alpha)_i (A_\beta)_j\,.   \label{Eq:Psi-RSh}
 \end{align}
\end{subequations}
Due to Theorem \ref{Thm:Pencil-MQV-bis}, the bivector $P+\psi_{\zbar}$ is always non-degenerate quasi-Poisson on $M_{\overline{Q}_d}^{\gamma}$ for the moment map \eqref{Eq:Momap-RS} (or equivalently \eqref{Eq:Momap-RS-2}). 
In particular, this yields a quasi-Poisson pencil centered at $P$ of order $d(d-1)/2$. 
After performing reduction, we obtain a Poisson pencil on the spin RS phase space $\Ccal_{n,d,q}^\times$ \eqref{Eq:Cndq}. 

\subsection{Local parametrization}  \label{S:sRS-loc}

We continue with the notations of the previous subsection. 
Following \cite[\S4.1]{CF}, consider the affine space $(\CC^\times)^n\times \CC^{2nd}$ with coordinates $(x_i,a_i^\alpha,b_i^\alpha)_{1\leq i\leq n}^{1\leq \alpha \leq d}$, and define the following subspace of dimension $2nd$: 
\begin{equation} \label{Eq:hreg}
 \h_\reg = \Big\{(x_i,a_i^\alpha,b_i^\alpha) \mid x_i\neq x_j,\, x_i\neq q x_j,\, \forall i\neq j;\quad 
 \sum_{1\leq \alpha \leq d} a_i^\alpha=1, \, \forall i; \quad \det(\mathcal{Z}_\alpha)\neq 0,\, \forall \alpha \Big\}\,,
\end{equation}
where for $\alpha\in \{0\}\cup \{1,\ldots,d\}$, we introduce the following matrix: 
\begin{equation*}
 \mathcal{Z}_\alpha\in \Mat_{n\times n}(\CC),\quad (\mathcal{Z}_\alpha)_{ij}:= q\frac{\sum_{\beta=1}^d a_i^\beta b_j^\beta}{x_ix_j^{-1}-q} + \sum_{\beta=1}^\alpha a_i^\beta b_j^\beta\,, \quad 1\leq i,j\leq n\,.
\end{equation*}
(The second sum is omitted when $\alpha=0$.) There is an action of the symmetric group $S_n$ on $\h_\reg$ by simultaneous permutation of the indices:  $\sigma \cdot (x_i,a_i^\alpha,b_i^\alpha) = (x_{\sigma(i)},a_{\sigma(i)}^\alpha,b_{\sigma(i)}^\alpha)$. 

\begin{proposition}[\cite{CF}]  \label{Prop:CF}
The morphism of affine varieties $\xi:\h_\reg/S_n \to \Ccal_{n,d,q}^\times$ obtained by mapping 
$(x_i,a_i^\alpha,b_i^\alpha)$ to the orbit of the point $(X,Z,A_\alpha,B_\alpha)$ given by 
\begin{equation} \label{Eq:xi-local}
 X_{ij}=\delta_{ij} x_i,\quad Z_{ij}=q\frac{\sum_{\beta=1}^d a_i^\beta b_j^\beta}{x_ix_j^{-1}-q}, \quad 
 (A_\alpha)_i=a_i^\alpha,\quad (B_\alpha)_i=b_i^\alpha\,,
\end{equation}
is well-defined and injective. Furthermore, $\xi$ is a morphism of Poisson varieties if $\Ccal_{n,d,q}^\times$ is endowed with the Poisson bracket defined by the Poisson bivector $P$ \eqref{Eq:qP-MQV-bis} and $\h_\reg/S_n$ is endowed with the Poisson bracket $\br{-,-}_0$ uniquely determined by 
 \begin{subequations} \label{Eqh}
 \begin{align}
  \br{x_i,x_j}_0=&0\,,\quad \br{x_i,  a_j^\alpha}_0=0\,,\quad 
\br{x_i,  b_j^\alpha}_0=\delta_{ij} x_i  b_j^\alpha \,, \label{Eqh1} \\
\br{ a_i^\alpha,  a_j^\beta}_0=&\frac12 \delta_{(i\neq j)}\frac{x_i+x_j}{x_i-x_j}
( a_i^\alpha  a_j^\beta + a_j^\alpha  a_i^\beta - a_j^\alpha  a_j^\beta - 
 a_i^\alpha  a_i^\beta ) +\frac12 o(\beta, \alpha)
( a_i^\alpha a_j^\beta + a_j^\alpha a_i^\beta ) \nonumber \\
&+\frac12 \sum_{\gamma=1}^d o(\alpha,\gamma) a_j^\beta 
( a_i^\alpha a_j^\gamma + a_j^\alpha a_i^\gamma )
-\frac12 \sum_{\gamma=1}^d o(\beta,\gamma) a_i^\alpha 
( a_j^\beta  a_i^\gamma+ a_i^\beta  a_j^\gamma)
\,, \label{Eqh2} \\
\br{ a_i^\alpha,  b_j^\beta}_0=& a_i^\alpha Z_{ij}-\delta_{\alpha\beta}Z_{ij}-
\frac12 \delta_{(i\neq j)}\frac{x_i+x_j}{x_i-x_j} ( a_i^\alpha- a_j^\alpha) b_j^\beta
+\delta_{(\alpha<\beta)} a_i^\alpha  b_j^\beta \nonumber \\
&+ a_i^\alpha \sum_{\gamma=1}^{\beta-1} a_i^\gamma ( b_j^\gamma- b_j^\beta) 
-\delta_{\alpha\beta} \sum_{\gamma=1}^{\beta-1}  a_i^\gamma  b_j^\gamma 
-\frac12 \sum_{\gamma=1}^d o(\alpha,\gamma) b_j^\beta 
( a_i^\alpha a_j^\gamma + a_j^\alpha a_i^\gamma ) \,, \label{Eqh3} \\
\br{ b_i^\alpha,  b_j^\beta}_0=&
\frac12 \delta_{(i\neq j)}\frac{x_i+x_j}{x_i-x_j} ( b_i^\alpha b_j^\beta  +  b_j^\alpha b_i^\beta) 
- b_i^\alpha Z_{ij} +  b_j^\beta Z_{ji} +\frac12 o(\beta,\alpha)
( b_i^\alpha b_j^\beta- b_j^\alpha  b_i^\beta) \nonumber \\
&- b_i^\alpha \sum_{\gamma=1}^{\beta-1} a_i^\gamma ( b_j^\gamma- b_j^\beta)
+ b_j^\beta \sum_{\gamma=1}^{\alpha-1} a_j^\gamma ( b_i^\gamma- b_i^\alpha) \,.\label{Eqh4}
 \end{align}
\end{subequations}
In those formulae, $o(-,-):\{1,\ldots,d\}^{\times 2}\to \{0,\pm 1\}$ is the skew-symmetric function such that $o(\alpha,\beta)=+1$ if $\alpha<\beta$, and $Z_{ij}$ is defined in \eqref{Eq:xi-local}. 
This induces an isomorphism of Poisson varieties $\xi:\h_\reg/S_n \to \Ccal^\circ$, where $\Ccal^\circ \subset \Ccal_{n,d,q}^\times$ is the image of $\xi$.
\end{proposition}

As the Poisson structure is non-degenerate, $\xi$ is an isomorphism of symplectic varieties. 
To establish that the map $\xi$ intertwines the two Poisson structures, it suffices to check the result on the functions
\begin{equation} \label{Eq:RS-functLoc}
 f_k:=\tr(X^k),\quad g_{\alpha\beta;k}:=\tr(A_\alpha B_\beta X^k), \quad k\geq0,\,\, 1\leq \alpha,\beta\leq d,
\end{equation}
as their pullbacks provide a local coordinate system in view of
\begin{equation} \label{Eq:xifg}
  \xi^\ast f_k:= \sum_i x_i^k\,, \quad
  \xi^\ast g_{\alpha \beta;k}=\sum_{i}a_i^\alpha b_i^\beta x_i^k\,, \quad
  \sum_{1\leq \alpha \leq d} \xi^\ast g_{\alpha \beta;k}=\sum_{i} b_i^\beta x_i^k\,.
\end{equation}

\subsubsection{Local expression of $\psi_{\zbar}$}
The argument leading to Proposition \ref{Prop:CF} can be adapted to the Poisson bivector $\psi_{\zbar}$ given in \ref{ss:RS-pencil} as follows.

\begin{lemma} \label{Lem:XiPsi}
 The isomorphism of varieties $\xi:\h_\reg/S_n \to \Ccal^\circ$ given by \eqref{Eq:xi-local} in Proposition \ref{Prop:CF} is an isomorphism of Poisson varieties if $\Ccal^\circ$ is endowed with the Poisson bracket defined by $\psi_{\zbar}$ and
 $\h_\reg/S_n$ is endowed with the Poisson bracket, denoted $\psi_{\zbar}^\loc$, uniquely determined by
 \begin{subequations} \label{Eq:psi}
 \begin{align}
\psi_{\zbar}^\loc(x_i,x_j)&=0\,,\quad \psi_{\zbar}^\loc(x_i,  a_j^\alpha)=0\,,\quad
\psi_{\zbar}^\loc(x_i,  b_j^\alpha )=0 \,, \label{Eq:psi1} \\
\psi_{\zbar}^\loc( a_i^\alpha,  a_j^\beta )&=  G_{\zbar}(i,j;\alpha,\beta) \, a_i^\alpha  a_j^\beta
\,, \label{Eq:psi2} \\
\psi_{\zbar}^\loc (a_i^\alpha,  b_j^\beta )&= -  G_{\zbar}(i,j;\alpha,\beta) \, a_i^\alpha  b_j^\beta
\,, \label{Eq:psi3} \\
\psi_{\zbar}^\loc (b_i^\alpha,  b_j^\beta)&=  G_{\zbar}(i,j;\alpha,\beta) \, b_i^\alpha  b_j^\beta
 \,. \label{Eq:psi4}
 \end{align}
\end{subequations}
Here, we have set for any $1\leq i,j\leq n$ and $1\leq \alpha,\beta \leq d$,
\begin{equation} \label{Eq:G-fct}
 G_{\zbar}(i,j;\alpha,\beta):= z_{\alpha \beta} + \sum_{1\leq \mu \leq d} (z_{\mu \alpha} a_j^\mu - z_{\mu \beta} a_i^\mu) + \sum_{1\leq \mu,\nu\leq d} z_{\mu \nu} a_i^\mu a_j^\nu\,,
\end{equation}
which is antisymmetric under simultaneously swapping the pairs $(i,\alpha)$ and $(j,\beta)$.
\end{lemma}
\begin{proof}
Gathering \eqref{Eq:Psi-RSa}, \eqref{Eq:Psi-RSg}--\eqref{Eq:Psi-RSh} and \eqref{Eq:RS-functLoc}, we compute
\begin{equation} \label{Eq:psi-fg}
 \psi_{\zbar}(f_k,f_l)=0,\,\,  \psi_{\zbar}(f_k,g_{\alpha\beta;l})=0, \quad
 \psi_{\zbar}(g_{\gamma \epsilon;k},g_{\alpha\beta;l})= (z_{\gamma \alpha}+ z_{\epsilon \beta} - z_{\gamma \beta} - z_{\epsilon \alpha}) \,g_{\gamma \epsilon;k}\, g_{\alpha\beta;l},
\end{equation}
for any $k,l\geq 1$ and $1\leq \alpha,\beta,\gamma,\epsilon \leq d$.
Checking that $\xi$ intertwines the two Poisson structures\footnote{A priori, $\psi_{\zbar}$ is only an antisymmetric biderivation on $\h_\reg/S_n$, and we deduce that it is a Poisson bracket because it is one on $\Ccal^\circ$ and $\xi$ is an isomorphism intertwining the two structures.} thanks to \eqref{Eq:xifg} is an exercise along the lines of Appendix A.4 in \cite{CF} (though much easier). 
\end{proof}

Note that $\psi_{\zbar}^\loc$ is well-defined because $\sum_{\alpha=1}^d G_{\zbar}(i,j;\alpha,\beta)a_i^\alpha=0$ for any $1\leq i\leq n$.

\begin{lemma} \label{Lem:Casim}
The Poisson bracket $\psi_{\zbar}^\loc$ is degenerate and admits the following Casimirs: 
\begin{equation} \label{Eq:Casim}
 x_1,\ldots,x_n,\quad a_j^\alpha b_j^\alpha, \text{ with }1\leq j \leq n,\,\, 1\leq \alpha \leq d\,.
\end{equation} 
\end{lemma}
\begin{proof}
To be degenerate, one needs that the algebra of Casimirs defined as 
$$\{ F\in \CC[\h_\reg]^{S_n} \mid \psi_{\zbar}^\loc(F,G)=0,\,\, \forall G\in \CC[\h_\reg]^{S_n} \}$$
has positive (Krull) dimension. Thus, it is direct if we establish that the elements \eqref{Eq:Casim} are Casimirs, or equivalently that they Poisson commute under $\psi^\loc_{\zbar}$ with the generators $(x_j,a_j^\alpha,b_j^\alpha)$. 
Any $x_i$ is clearly a Casimir by \eqref{Eq:psi1}. 
Next, for arbitrary $1\leq j \leq n$, $1\leq \beta \leq d$, one has  
\begin{equation*}
 \psi^\loc_{\zbar}(a_i^\alpha b_i^\alpha,a_j^\beta)
 = \psi^\loc_{\zbar}(a_i^\alpha,a_j^\beta) b_i^\alpha + a_i^\alpha \psi^\loc_{\zbar}(b_i^\alpha,a_j^\beta) 
 = (G_{\zbar}(i,j;\alpha,\beta)+G_{\zbar}(j,i;\beta,\alpha)) \, a_i^\alpha b_i^\alpha a_j^\beta
 =0,
\end{equation*}
by \eqref{Eq:psi2}--\eqref{Eq:psi3} and the antisymmetry of \eqref{Eq:G-fct}. 
Similarly $\psi^\loc_{\zbar}(a_i^\alpha b_i^\alpha,b_j^\beta)=0$. Thus the element $a_i^\alpha b_i^\alpha$ is also a Casimir. 
\end{proof}

By reduction of the quasi-Poisson pencil given in \ref{ss:RS-pencil}, 
any bivector $z_0 P+\psi_{\zbar}$ on $\Ccal_{n,d,q}^\times$ with $z_0\in \CC$ is Poisson, and it is non-degenerate whenever $z_0\neq 0$, cf. Corollary \ref{Cor:Pencil-redMQV}. Restricting the Poisson pencil to $\Ccal^\circ$, the isomorphism $\xi:\h_\reg/S_n \to \Ccal^\circ$ that intertwines the Poisson structures $\{-,-\}_0$ with $P$ (by Proposition \ref{Prop:CF}) and $\psi_{\zbar}^\loc$ with  $\psi_{\zbar}$ (by Lemma \ref{Lem:XiPsi}) is therefore intertwining the structure 
$z_0 \{-,-\}_0+\psi_{\zbar}^\loc$ with $z_0 P + \psi_{\zbar}$, making the former a Poisson bracket for any possible parameter. These observations can be summarized in the following form. 

\begin{corollary} \label{Cor:RS-rank}
There is a Poisson pencil on $\h_\reg/S_n$ spanned by the Poisson brackets 
$\{-,-\}_0$ and $\psi_{\zbar}^\loc$, with $\zbar=(z_{\alpha \beta})_{1\leq \alpha<\beta\leq d}$ defined as in \ref{ss:RS-pencil}. Any linear combination $z_0 \{-,-\}_0 + \psi_{\zbar}^\loc$ with $z_0\in \CC$ is therefore a Poisson bracket, which is non-degenerate when $z_0\neq 0$. 
This pencil (and therefore the pencil spanned by $P$ and $\psi_{\zbar}$ on $\Ccal_{n,d,q}^\times$) has order $r\geq \frac{(d-1)(d-2)}{2}+1$.  
\end{corollary}
\begin{proof}
We only need to compute the bound on the order of the pencil. 
Since each $\psi_{\zbar}^\loc$ is degenerate while $\{-,-\}_0$ is not, 
the order equals $\mathrm{ord}(\psi_{\zbar}^\loc)+1$. 

Assume that for some $\zbar$, $\psi_{\zbar}^\loc\equiv 0$ identically.  
By the definition of $\h_\reg$ \eqref{Eq:hreg}, we can find a point $m\in \h_\reg$ where all $(b_i^\alpha)$ are nonzero while  $a_i^\alpha=\delta_{\alpha,1}$. At such a point, \eqref{Eq:psi4} reads  
\begin{equation} \label{Eq:Pf-Rank1}
 \psi_{\zbar,m}^\loc (b_i^\alpha,  b_j^\beta)=  
\left[z_{\alpha \beta} + z_{1 \alpha}- z_{1\beta} \right] \, b_i^\alpha  b_j^\beta\,, \quad 1\leq \alpha,\beta \leq d\,,
\end{equation}
as $z_{11}=0$ by skewsymmetry of the matrix $(z_{\alpha \beta})$. By assumption, \eqref{Eq:Pf-Rank1} vanishes and we get that 
\begin{equation} \label{Eq:Pf-Rank2}
 z_{\alpha \beta} = z_{\alpha 1}+ z_{1\beta} \,,
\end{equation}
for any $1\leq \alpha,\beta \leq d$. Hence we can fix arbitrarily $z_{12},\ldots,z_{1d}$ then define the other parameters through \eqref{Eq:Pf-Rank2} and still get $\psi_{\zbar,m}^\loc=0$. As we have freedom in choosing parameters, 
$$ \mathrm{ord}(\psi_{\zbar,m}^\loc) = \frac{d(d-1)}{2}- (d-1) = \frac{(d-1)(d-2)}{2}\,.$$
It follows by definition that $\mathrm{ord}(\psi_{\zbar}^\loc) \geq \mathrm{ord}(\psi_{\zbar,m}^\loc)$, and we get the claimed bound.
\end{proof}

\begin{remark} 
We expect that the inequality for the order given in Corollary \ref{Cor:RS-rank} is an equality. 
This is the case for $d=2$, as the order is precisely $1$. To see this, note that $\psi_{\zbar}^{\loc}\equiv 0$ identically for any $z_{12}\in \CC$  
because $G_{\zbar}(i,j;\alpha,\beta)$ \eqref{Eq:G-fct} is always vanishing on $\h_{\reg}$.
\end{remark}

\subsubsection{Comparison with the structure of Arutyunov-Olivucci}

In \cite{AO}, the spin RS phase space $\Ccal_{n,d,q}^\times$ \eqref{Eq:Cndq} is obtained by reduction using Poisson-Lie symmetries. Their local parametrization, presented in \cite[\S5]{AO} (with $\ell=d$), is related to the one of Proposition \ref{Prop:CF} by substituting their matrices 
\begin{equation}
 Q,\,L \in \Gl_n(\CC), \quad \mathbf{a} \in \Mat_{n\times d}(\CC), \quad \mathbf{c} \in \Mat_{d\times n}(\CC),
\end{equation}
respectively with 
\begin{equation}
 X, \, \kappa XZX^{-1} \in \Gl_n(\CC), \quad (A_1 \cdots A_d) \in \Mat_{n\times d}(\CC), \quad 
 (B_1^T \cdots B_d^T)^T \in \Mat_{d\times n}(\CC)\,.
\end{equation}
They have an additional parameter $\kappa\neq 0$ which can be omitted by simultaneously relabeling $\mathbf{a}=\sqrt{\kappa} \mathbf{a}$, $\mathbf{c}=\sqrt{\kappa} \mathbf{c}$ and rescaling their Poisson brackets by $\kappa^{-1}$. 
Thus, we can take $\kappa=1$, and write down the 2 Poisson brackets that they derived locally on $\h_{\reg}/S_n$ (see\footnote{These expressions appear in compact $r$-matrix notation in \cite{AO}, and their presentation in \cite{Ar} contains typos.} \cite[(6.4)]{AO} and \cite{Ar}) by
\begin{subequations} \label{Eq:AO}
 \begin{align}
 \br{x_i,x_j}_{\pm}=&0\,,\quad \br{x_i,  a_j^\alpha}_{\pm}=0\,,\quad 
\br{x_i,  b_j^\alpha}_{\pm}=\delta_{ij} x_i  b_j^\alpha \,, \label{EqAO1} \\
\br{ a_i^\alpha,  a_j^\beta}_{\pm}=&
\frac12 \delta_{(i\neq j)}\frac{x_i+x_j}{x_i-x_j}
( a_i^\alpha  a_j^\beta + a_j^\alpha  a_i^\beta - a_j^\alpha  a_j^\beta -  a_i^\alpha  a_i^\beta ) 
\mp \frac12 o(\beta, \alpha)  a_j^\alpha a_i^\beta \nonumber \\
&\mp \frac12 \sum_{\gamma=1}^d o(\alpha,\gamma) a_j^\beta  a_j^\alpha a_i^\gamma
\pm \frac12 \sum_{\gamma=1}^d o(\beta,\gamma) a_i^\alpha  a_i^\beta  a_j^\gamma 
 \mp\frac12 \left( \sum_{\gamma,\epsilon} o(\epsilon,\gamma) a_i^\epsilon  a_j^\gamma\right)  a_i^\alpha  a_j^\beta\,, \label{EqAO2} \\
\br{ a_i^\alpha,  b_j^\beta}_{\pm}=&
 a_i^\alpha Z_{ij}-\delta_{\alpha\beta}Z_{ij}
-\frac12 \delta_{(i\neq j)}\frac{x_i+x_j}{x_i-x_j} ( a_i^\alpha- a_j^\alpha) b_j^\beta \nonumber 
+ a_i^\alpha \sum_{\gamma\gtrless \beta} a_i^\gamma  b_j^\gamma 
-\delta_{\alpha\beta} \sum_{\gamma\gtrless \beta}  a_i^\gamma  b_j^\gamma \nonumber \\
&\pm\frac12 \sum_{\gamma=1}^d o(\alpha,\gamma) b_j^\beta  a_j^\alpha a_i^\gamma 
\pm\frac12 \left( \sum_{\gamma,\epsilon} o(\epsilon,\gamma) a_i^\epsilon  a_j^\gamma\right)  a_i^\alpha  b_j^\beta 
 +\frac12  a_i^\alpha  a_i^\beta  b_j^\beta - \frac12 \delta_{\alpha \beta} a_i^\alpha  b_j^\beta \,, \label{EqAO3} \\
\br{ b_i^\alpha,  b_j^\beta}_{\pm}=&
\frac12 \delta_{(i\neq j)}\frac{x_i+x_j}{x_i-x_j} ( b_i^\alpha b_j^\beta  +  b_j^\alpha b_i^\beta) 
- b_i^\alpha Z_{ij} +  b_j^\beta Z_{ji} \pm \frac12 o(\beta,\alpha)  b_j^\alpha  b_i^\beta \nonumber \\
&- b_i^\alpha \sum_{\gamma\gtrless \beta} a_i^\gamma  b_j^\gamma
+ b_j^\beta \sum_{\gamma\gtrless \alpha} a_j^\gamma  b_i^\gamma 
\mp\frac12 \left( \sum_{\gamma,\epsilon} o(\epsilon,\gamma) a_i^\epsilon  a_j^\gamma\right)  b_i^\alpha  b_j^\beta 
+\frac12  ( a_j^\alpha -  a_i^\beta) b_i^\alpha  b_j^\beta \,, \label{EqAO4}
 \end{align}
\end{subequations}
where we use the notation from Proposition \ref{Prop:CF}. 
Observe that both Poisson brackets are equivalent under relabeling spin indices through $\alpha \mapsto d+1-\alpha$ (this was remarked before reduction in \cite{FF21}). 
Hence we shall only focus on the minus Poisson bracket. 

\begin{proposition} \label{Prop:pencilRS}
 The minus Poisson bracket $\br{-,-}_-$ constructed by Arutyunov and Olivucci \cite{AO} and the Poisson bracket $\br{-,-}_0$ constructed by Chalykh and Fairon \cite{CF} belong to the pencil of compatible Poisson brackets from Corollary \ref{Cor:RS-rank}. 
\end{proposition}
\begin{proof}
By definition, $\br{-,-}_0$ is certainly in the pencil. 
Take the sequence of parameters $\zbar^-$ given by $z_{\alpha\beta}^-=\frac12$, $\alpha<\beta$. The corresponding skewsymmetric matrix has entries $\frac12 o(\alpha,\beta)$, and the Poisson structure $\psi_{\zbar^-}^\loc$ \eqref{Eq:psi} is given in terms of
\begin{equation}
 G_{\zbar^-}(i,j;\alpha,\beta)= \frac12 o(\alpha,\beta)
 + \frac12 \sum_{\gamma}  o(\gamma,\alpha) a_j^\gamma  -  \frac12 \sum_{\gamma}  o(\gamma,\beta) a_i^\gamma
 +\frac12 \sum_{\gamma,\epsilon} o(\epsilon,\gamma) a_i^\epsilon a_j^\gamma\,.
\end{equation}
It is then an exercise to check that $\br{-,-}_-=\br{-,-}_0 + \psi_{\zbar^-}^\loc$ using \eqref{Eqh} and \eqref{Eq:AO}, remembering that $\sum_\gamma a_j^\gamma=1$ for each $1\leq j\leq n$.
\end{proof}

\begin{remark}
In \cite{CF}, it is conjectured that the irreducible component of $\Ccal_{n,d,q}^\times$ containing $\Ccal^\circ=\xi(\h_\reg/S_n)$ is the whole variety $\Ccal_{n,d,q}^\times$, in analogy with the case $d=1$ due to Oblomkov \cite{Ob}. 
Under that conjecture, we can reformulate Proposition \ref{Prop:pencilRS} as the compatibility of the Poisson structures constructed in \cite{AO,CF} on the whole spin RS phase space $\Ccal_{n,d,q}^\times$. 
\end{remark}

Finally, we can use Proposition \ref{Prop:pencilRS} to provide another quiver interpretation of the minus Poisson structure of Arutyunov and Olivucci \cite{AO}. 
To do so, we need to unwind the construction of the master phase space $M_{\overline{Q}_d}^{\gamma}$ considered in \ref{S:sRS-Gen}. At a purely geometric level, the quasi-Poisson structure is obtained by considering the internally fused double $\Gl_n\times \Gl_n = \{(X,Z)\}$ as a $\Gl_n$-variety, with $d$ copies of the space 
$(T^\ast \CC^n)^\circ:=\{(W,V)\in T^\ast \CC^n \mid 1+VW\neq 0\}$ seen as a $(\Gl_n \times \CC^\times)$-variety, and we perform fusion of these different actions in a specific order. 
Observe from Remark \ref{Rem:Cstar-Fus} that the fusion terms added by the $\CC^\times$ actions will, in fact, appear as  bivectors $\psi_{\zbar}$ in the pencil described in \ref{ss:RS-pencil}. 
By tracking all these fusion terms, we see that they correspond to adding $-\psi_{\zbar^-}$, for $\zbar^-$ given by $z_{\alpha\beta}^-=\frac12$, $\alpha<\beta$. 
It follows from the proof of Proposition \ref{Prop:pencilRS}  that these are precisely the terms that are killed to go from $\br{-,-}_0$ (the local expression of the bivector $P$) to $\br{-,-}_-$ (the Poisson structure from \cite{AO}). 
Forgetting the fusion of these $\CC^\times$ actions can be reinterpreted as follows. 
We consider $\overline{Q}_d^{\infty_d}$, where $Q_d^{\infty_d}$ consists of the vertex set $\{0,\infty_1,\ldots,\infty_d\}$ and the arrows 
$x:0\to 0$, $v_\alpha:\infty_\alpha \to 0$ for $1\leq \alpha \leq d$. We then perform the reduction of 
$M_{\overline{Q}_d^{\infty_d}}^{(0,1,\ldots,1)}(n,1,\ldots,1)$ at $q\Id_n \in \Gl_n(\CC)$, i.e. we consider $\infty_1,\ldots,\infty_d$ as `framing vertices' with respect to which we are not reducing\footnote{We use the standard trick that a (Hamiltonian) quasi-Poisson $(G\times H)$-variety where $H$ is abelian can be viewed as a (Hamiltonian) quasi-Poisson $G$-variety. Here $H=(\CC^\times)^d$ and $G=\Gl_n(\CC)$.}. In this way, Van den Bergh's quasi-Poisson structure associated with $Q_d^{\infty_d}$ yields the minus Poisson structure of \cite{AO} in local coordinates.

\subsubsection{Hamiltonian formulation of the spin RS system}

It is proved in \cite{AO,CF} that, with $h:=2(q^{-1}-1) \tr(Z)$, the equations of motion for the derivation $\frac{d}{dt}:=\br{h,-}$ (where $\br{-,-}$ denotes the Poisson bracket from the corresponding reference) satisfy locally
\begin{subequations}
 \begin{align}
 & \frac{dx_i}{dt} =  2{f}_{ii}\,x_i \,,\quad
\frac{d a_i^\alpha}{dt} =
\sum_{k\neq i} V_{ik}{f}_{ik} (a_k^\alpha-a_i^\alpha) \,,\quad
\frac{d b_i^\alpha}{dt} =\sum_{k\neq i}
\left(\,V_{ik}{f}_{ik} b_i^\alpha -V_{ki}{f}_{ki} b_k^\alpha  \,\right)\,, \label{Trigc} \\
&\text{where}\quad V_{ik}=\frac{x_i+x_k}{x_i-x_k}-\frac{x_i+qx_k}{x_i-qx_k}\,, \quad f_{ij}=\sum_{1\leq \alpha \leq d} a_i^\alpha b_j^\alpha \,. \label{EqPotV}
 \end{align}
\end{subequations}
These are the equations of the spin RS system introduced by Krichever and Zabrodin \cite{KrZ} under the constraints $\sum_\alpha a_i^\alpha=1$.

\begin{proposition} \label{Prop:HamRS}
 For any $\zbar \in \CC^{d(d-1)/2}$, let $\br{-,-}:=\br{-,-}_0+\psi^\loc_{\zbar}$. Then the equations \eqref{Trigc} hold for $h:=2(q^{-1}-1) \tr(Z)=2\sum_i f_{ii}$.
In particular, any non-degenerate Poisson bracket from the pencil on $\h_\reg/S_n$ gives the equations \eqref{Trigc}, up to rescaling of $h$.
\end{proposition}
\begin{proof}
Note that $\psi^\loc_{\zbar}(h,-)=0$ because $h$ is a linear combination of Casimirs by Lemma \ref{Lem:Casim}. Thus $\br{h,-}=\br{h,-}_0$, and \eqref{Trigc} holds for $\br{-,-}_0$ by \cite{CF}.

For the second part, recall that a non-degenerate Poisson bracket from the pencil is of the form $\br{-,-}=z_0\br{-,-}_0+\psi^\loc_{\zbar}$ with $z_0\in \CC^\times$. In that case we get \eqref{Trigc} for $\frac{d}{dt}=\br{z_0^{-1}h,-}$. 
\end{proof}

\begin{remark}
In the pencil introduced in \ref{ss:RS-pencil}, we assumed that the parameters written in terms of the arrows of $Q_d$ satisfy $z_{x,v_\alpha}=0$ for all $1\leq \alpha \leq d$. If these parameters were taken to be nonzero, we would get a larger pencil, but the proof of Proposition \eqref{Prop:HamRS} would no longer be valid. To illustrate this claim, note that for such a case the second equality in \eqref{Eq:psi-fg} gets replaced with
$\psi_{\zbar}(f_k,g_{\alpha\beta;l})=k (z_{x,v_\alpha}-z_{x,v_\beta})f_k g_{\alpha\beta;l}$. 
In particular one may deduce $\psi_{\zbar}^\loc(x_i,b_j^\alpha)=(z_{x,v_\alpha}-z_{x,v_\beta}) x_i$, which is nonzero if the parameters $(z_{x,v_1},\ldots,z_{x,v_d})$ are not all equal. 
Similarly, the second item of Proposition \ref{Pr:IntRS} below may no longer be valid for the extended pencil because $\mathcal{H}$ may not lie in the Poisson center of $\mathcal{Q}$.
\end{remark}

\subsection{Integrability} \label{S:sRS-Int}

For any $k\geq 1$ and $1\leq \alpha,\beta \leq d$, introduce the following functions, 
\begin{align*}
 h_k:= \tr(Z^k)\,, \quad t_{k;\alpha \beta}:=\tr(W_\alpha V_\beta Z^k)\,, \quad 
 h_{k;\alpha}:=\tr(\mathcal{Z}_\alpha^k)\,,
\end{align*}
where $\mathcal{Z}_\alpha$ was defined by \eqref{Eq:Rel-VWAB}. 
We form the following commutative subalgebras of $\CC[\Ccal_{n,d,q}^\times]$, 
\begin{align*}
 \mathcal{H}:=&\CC[h_k \mid k\geq 1]\,, \quad 
 \mathcal{Q}:=\CC[t_{k;\alpha \beta} \mid k\geq 1,\,\, 1\leq \alpha,\beta \leq d] \,, \\
 \mathcal{H}_{\mathrm{int}}:=&\CC[h_{k;\alpha} \mid k\geq 1,\,\, 1\leq \alpha \leq d]\,.
\end{align*}
Recall from \cite[\S5]{CF} the following chain of inclusions: $\mathcal{H} \subset \mathcal{H}_{\mathrm{int}} \subset \mathcal{Q}$. 

\begin{proposition} \label{Pr:IntRS}
For any Poisson structure $P+\psi_{\zbar}$ on $\Ccal_{n,d,q}^\times$ with $\zbar$ as in \ref{ss:RS-pencil}, the commutative algebras $\mathcal{H},\mathcal{H}_{\mathrm{int}},\mathcal{Q}$ are Poisson algebras. 
Furthermore, if we denote by $\Ccal^\times$ the irreducible component of $\Ccal_{n,d,q}^\times$ containing $\Ccal^\circ=\xi(\h_\reg/S_n)$: 
 \begin{enumerate}
  \item $\mathcal{H}_{\mathrm{int}}$ is an abelian Poisson algebra of dimension $nd=\frac12\dim(\Ccal^\times)$, hence it defines an integrable system on $\Ccal^\times$;
  \item $\mathcal{Q}$ has codimension $n$ in $\Ccal^\times$ with Poisson center containing $\mathcal{H}$ of dimension $n$ on $\Ccal^\times$, hence it defines a degenerately integrable system on $\Ccal^\times$.
 \end{enumerate}
\end{proposition}
\begin{proof}
When the parameters in $\zbar$ are all zero, this follows from Theorems 2.4 and 5.5 in \cite{CF}. In particular, the stated dimensions always hold since they do not depend on the Poisson structure. 

We can compute in general that $\psi_{\zbar}(h_{k;\alpha},h_{l;\beta})=0$ for any indices 
because $\psi_{\zbar}(Z_{ij},-)=0$ and $\psi_{\zbar}((W_\gamma V_\gamma)_{ij},-)=0$ by \eqref{Eq:Psi-RS}. 
Thus $\mathcal{H}_{\mathrm{int}}$ and its subalgebra $\mathcal{H}$ stay abelian for any Poisson structure. 
Using \eqref{Eq:Psi-RS} again, we find as in \eqref{Eq:psi-fg}, 
\begin{equation}
 \psi_{\zbar}(t_{k;\gamma \epsilon},t_{l;\alpha \beta})
 =(z_{\gamma \alpha}+ z_{\epsilon \beta} - z_{\gamma \beta} - z_{\epsilon \alpha}) \,
 t_{k;\gamma \epsilon}\,t_{l;\alpha \beta},
\end{equation}
so that $\mathcal{Q}$ is closed under $\psi_{\zbar}$, hence it is also closed under $P+\psi_{\zbar}$. 
\end{proof}


\appendix

\section{Explicit correspondence} \label{App:Corr}

We verify the correspondence \eqref{Eq:corrPOm} for Theorem \ref{Thm:qPqHam-MQV-bis}. 
Since the structures are obtained by fusion (cf. Remark \ref{Rem:MQV-fusion} and \cite{VdB1,VdB2,Ya}), it follows from (the complex algebraic analogue of) \cite[Prop. 10.7]{AKSM} that we only need to check the correspondence for a $1$-arrow quiver. 
So take $Q= 1 \stackrel{a}{\longrightarrow} 2$. For the dimension vector $\nfat=(n_1,n_2)$ and parameter $\gamma:=\gamma_a\in \CC$, the data $(\Xtt_a,\Xtt_{a^\ast})$ with 
\begin{equation}
 \Xtt_a=\left( 
 \begin{array}{cc}
  0_{n_1\times n_1}& A \\
  0_{n_1\times n_2}&0_{n_2\times n_2}
\end{array} \right) \,, \quad 
\Xtt_{a^\ast}=\left( 
 \begin{array}{cc}
  0_{n_1\times n_1}&  0_{n_2\times n_1} \\
 B&0_{n_2\times n_2}
\end{array} \right) \,,
\end{equation}
for $A\in \Mat_{n_1\times n_2}(\CC)$, $B\in \Mat_{n_2\times n_1}(\CC)$, 
parametrize a point $\Xtt\in M_{\overline{Q}}$. 
The open subspace $M_{\overline{Q}}^\gamma$ corresponds to requiring $\det(\gamma \Id_{n_1}+AB)\neq 0$, 
and we define $\Phi^\gamma:M_{\overline{Q}}^\gamma\to \Gl_{n_1}\times \Gl_{n_2}$ by 
\begin{equation} 
 \Phi^\gamma(\Xtt)=(\gamma \Id+\Xtt_a \Xtt_{a^\ast})(\gamma\Id+ \Xtt_{a^\ast}\Xtt_a)^{-1}
 =\left( 
 \begin{array}{cc}
  \gamma\Id_{n_1}+AB&  0_{n_2\times n_1} \\
0_{n_1\times n_2} &(\gamma\Id_{n_2}+BA)^{-1}
\end{array} \right)\,.
\end{equation}
If $\gamma=0$, $n_1=n_2$, we are in the case of the quasi-Poisson double of $\Gl_{n_1}(\CC)$ where the correspondence can be found in \cite[Ex.~10.5]{AKSM}.
Hence we can work with $\gamma\neq 0$ from which, up to rescaling, we can assume that $\gamma=1$; therefore we simply write $M_{\overline{Q}}^\gamma,P^\gamma, \omega^\gamma,\Phi^\gamma$ as $M_{\overline{Q}}^\circ,P, \omega,\Phi $ hereafter. 
Let us already note that the matrix-valued vector fields $\partial_a, \partial_{a^\ast}$ take the form 
\begin{equation}
 \partial_a=\left( 
 \begin{array}{cc}
  0_{n_1\times n_1}& 0_{n_2\times n_1} \\
  (\partial/\partial A_{ji})_{ij}&0_{n_2\times n_2}
\end{array} \right) \,, \quad 
\partial_{a^\ast}=\left(
 \begin{array}{cc}
  0_{n_1\times n_1}&  (\partial/\partial B_{ji})_{ij}  \\
 0_{n_1\times n_2}&0_{n_2\times n_2}
\end{array} \right) \,,
\end{equation}
from which we deduce $\partial_a \Xtt_{a^\ast}=\Xtt_{a^\ast}\partial_a=0$ 
and $\partial_{a^\ast} \Xtt_a=\Xtt_a \partial_{a^\ast}=0$. 

Take the dual bases $(E_{ij})_{ij}$ and $(E^{ij}=E_{ji})_{ji}$ of $\gl_{n_1}\times \gl_{n_2}$, where $(i,j)\in \{1,\ldots,n_1\}^{\times 2} \cup \{n_1+1,\ldots,n_1+n_2\}^{\times 2}$ and $E_{ij}$ is the elementary matrix with only nonzero entry in position $(i,j)$. We want to verify \eqref{Eq:corrPOm}, i.e. for any $X\in TM_{\overline{Q}}^\circ$
\begin{equation} \label{Eq:App1}
 P^\sharp \circ \omega^\flat(X) = X- \frac14 \sum_{i,j} \langle X, (\Phi^{-1} \dd\Phi-\dd\Phi\, \Phi^{-1})_{ij}\rangle \, (E_{ij})_{M_{\overline{Q}}^\circ} \,.
\end{equation}
We do so for any $X=(\partial_a)_{ij}=\partial/\partial a_{ji}$ and leave the case $X=(\partial_{a^\ast})_{ij}$ as an exercise; this is sufficient to prove our claim. 
To evaluate the left-hand side of \eqref{Eq:App1}, we compute using \eqref{Eq:omega-MQV}
\begin{align*}
 \omega^\flat((\partial_a)_{ij})&=
 -\frac12 \tr\left[(\Id+\Xtt_a \Xtt_{a^\ast})^{-1} \, \dd\Xtt_a \wedge \dd\Xtt_{a^\ast} 
  + \dd\Xtt_a \wedge (\Id+\Xtt_{a^\ast}\Xtt_a)^{-1} \, \dd\Xtt_{a^\ast} \right]^\flat((\partial_a)_{ij})  \\
&=
 -\frac12 \left[\dd\Xtt_{a^\ast} (\Id+\Xtt_a \Xtt_{a^\ast})^{-1} 
  + (\Id+\Xtt_{a^\ast}\Xtt_a)^{-1} \, \dd\Xtt_{a^\ast} \right]_{ij}\,.
\end{align*}
Thus, we get from \eqref{Eq:qP-MQV} 
\begin{align*}
P^\sharp \circ \omega^\flat((\partial_a)_{ij})
&=\frac12 
 \tr\left[(\Id+\Xtt_{a^\ast}\Xtt_a) \, \partial_a \wedge \partial_{a^\ast}
 +  \partial_a \wedge (\Id+\Xtt_a \Xtt_{a^\ast}) \partial_{a^\ast}\right]^\sharp \circ \omega^\flat((\partial_a)_{ij}) \\
 &= \frac12 (\partial_a)_{ij} 
 + \frac14 \left[ (\Id+\Xtt_{a^\ast}\Xtt_a)^{-1} \partial_a (\Id+\Xtt_a\Xtt_{a^\ast}) 
                + (\Id+\Xtt_{a^\ast}\Xtt_a) \partial_a (\Id+\Xtt_a \Xtt_{a^\ast})^{-1}\right]_{ij}\,.
\end{align*}

Next, we look at the right-hand side of \eqref{Eq:App1}. From the general equalities
\begin{subequations}
\begin{align}
 \langle (\partial_b)_{ij}, \dd(\Id+\Xtt_b \Xtt_{b^\ast})^{\epsilon(b)}_{uv}\rangle
&=\left\{
\begin{array}{lc}
 \delta_{uj} (\Xtt_{b^\ast})_{iv} & \epsilon(b)=+1 ,\\
 - (\Id+\Xtt_b \Xtt_{b^\ast})^{-1}_{uj} (\Xtt_{b^\ast}(\Id+\Xtt_b \Xtt_{b^\ast})^{-1})_{iv} & \epsilon(b)=-1,
\end{array}
\right.\\
 \langle (\partial_b)_{ij}, \dd(\Id+\Xtt_{b^\ast} \Xtt_b)^{-\epsilon(b)}_{uv}\rangle
&=\left\{
\begin{array}{ll}
 -((\Id+\Xtt_{b^\ast} \Xtt_b)^{-1} \Xtt_{b^\ast})_{uj} (\Id+\Xtt_{b^\ast} \Xtt_b)^{-1}_{iv} & \epsilon(b)=+1 ,\\
 (\Xtt_{b^\ast})_{uj} \delta_{iv}  & \epsilon(b)=-1,
\end{array}
\right.
\end{align}
\end{subequations}
(valid for any $b\in \overline{Q}$ and any quiver $Q$),
we deduce for the case at hand
\begin{align*}
 \langle (\partial_a)_{ij}, (\Phi^{-1} \dd\Phi-\dd\Phi\, \Phi^{-1})_{uv}\rangle
 =&(\Id+\Xtt_a \Xtt_{a^\ast})^{-1}_{uj} (\Xtt_{a^\ast})_{iv} - (\Xtt_{a^\ast})_{uj} (\Id+\Xtt_{a^\ast} \Xtt_a)^{-1}_{iv} \\
 &-\delta_{uj} (\Xtt_{a^\ast}(\Id+\Xtt_a \Xtt_{a^\ast})^{-1})_{iv}
 +((\Id+\Xtt_{a^\ast} \Xtt_a)^{-1} \Xtt_{a^\ast})_{uj}  \delta_{iv}\,.
\end{align*}
Together with \eqref{Eq:Act-inf} for $\xi=E_{ij}$, we obtain
\begin{align*}
 \eqref{Eq:App1}_{RHS}
&=(\partial_a)_{ij} - \frac14
 \sum_{b=a,a^\ast}
 \Big[  \Xtt_{a^\ast} (\partial_{b} \Xtt_{b} -  \Xtt_b \partial_b) (\Id+\Xtt_a \Xtt_{a^\ast})^{-1}
+ (\partial_{b} \Xtt_{b} -  \Xtt_b \partial_b) (\Id+\Xtt_{a^\ast} \Xtt_a)^{-1} \Xtt_{a^\ast}  \\
& \qquad \qquad \qquad - (\Id+\Xtt_{a^\ast} \Xtt_a)^{-1} (\partial_{b} \Xtt_{b} -  \Xtt_b \partial_b) \Xtt_{a^\ast}
- \Xtt_{a^\ast}(\Id+\Xtt_a \Xtt_{a^\ast})^{-1} (\partial_{b} \Xtt_{b} -  \Xtt_b \partial_b)  \Big]_{ij} \\
&=(\partial_a)_{ij} - \frac14 
 \Big[  \Xtt_{a^\ast} (\partial_{a^\ast} \Xtt_{a^\ast} -  \Xtt_a \partial_a) (\Id+\Xtt_a \Xtt_{a^\ast})^{-1}
+ (\partial_{a} \Xtt_{a} -  \Xtt_{a^\ast} \partial_{a^\ast}) \Xtt_{a^\ast} (\Id+\Xtt_a \Xtt_{a^\ast})^{-1}   \\
&\qquad \qquad \qquad - (\Id+\Xtt_{a^\ast} \Xtt_a)^{-1} (\partial_a \Xtt_a -  \Xtt_{a^\ast} \partial_{a^\ast}) \Xtt_{a^\ast}
- (\Id+\Xtt_{a^\ast} \Xtt_a)^{-1} \Xtt_{a^\ast} (\partial_{a^\ast} \Xtt_{a^\ast} -  \Xtt_a \partial_a)  \Big]_{ij}\,.
\end{align*}
It is straightforward to conclude after simplifying the second term as  
$$-\frac12 (\partial_a)_{ij} 
 + \frac14 \left[ (\Id+\Xtt_{a^\ast}\Xtt_a)^{-1} \partial_a (\Id+\Xtt_a\Xtt_{a^\ast}) 
+ (\Id+\Xtt_{a^\ast}\Xtt_a) \partial_a (\Id+\Xtt_a \Xtt_{a^\ast})^{-1}\right]_{ij}\,.$$


\Addresses

\end{document}